\documentclass[12pt,amsymb,fullpage]{amsart}
\usepackage{amssymb,amscd,pstricks}

\newtheorem{theorem}{Theorem}[section]
\newtheorem{defn}[theorem]{Definition}

\newtheorem{lemma}[theorem]{Lemma}

\newtheorem{eple}[theorem]{Example}
\newtheorem{rmk}[theorem]{Remarks}
\newtheorem{dsc}[theorem]{Discussion}
\newtheorem{nota}[theorem]{Notation}

\newsavebox{\indbin}
\savebox{\indbin}{\begin{picture}(0,0)
\newlength{\gnu}
\settowidth{\gnu}{$\smile$} \setlength{\unitlength}{.5\gnu}
\put(-1,-.65){$\smile$} \put(-.25,.1){$|$}
\end{picture}}

\newcommand{\be}{\begin{enumerate}}
\newcommand{\bd}{\begin{defn}}
\newcommand{\bt}{\begin{theorem}}
\newcommand{\bl}{\begin{lemma}}
\newcommand{\ee}{\end{enumerate}}
\newcommand{\ed}{\end{defn}}
\newcommand{\et}{\end{theorem}}
\newcommand{\el}{\end{lemma}}

\begin{document}
\title{Some Arguments for the Wave Equation in Quantum Theory}
\author{Tristram de Piro}
\address{Flat 3, Redesdale House, 85 The Park, Cheltenham, GL50 2RP}
 \email{t.depiro@curvalinea.net}
\thanks{}
\begin{abstract}
We clarify some arguments concerning Jefimenko's equations, as a way of constructing solutions to Maxwell's equations, for charge and current satisfying the continuity equation. We then isolate a condition on non-radiation in all inertial frames, which is intuitively reasonable for the stability of an atomic system, and prove that the condition is equivalent to the charge and current satisfying certain relations, including the wave equations. Finally, we prove that with these relations, the energy in the electromagnetic field is quantised and displays the properties of the Balmer series.
\end{abstract}
\maketitle

This paper is divided into three parts. The first part deals with some technical issues concerning Jefimenko's equations, which are perhaps not completely clear from \cite{G}. Namely that we can obtain solutions to Maxwell's equations from a given charge and current configuration $(\rho,\overline{J})$, satisfying the continuity equation, with $\overline{J}$ vanishing at infinity, using Jefimenko's equation to define the electric and magnetic fields $\{\overline{E},\overline{B}\}$. In particular, this is the case for localised charge and current configurations, or when the charge and current decays rapidly at infinity, a condition true for the Schwartz class of functions which I have denoted by $S(\overline{R}^{3})$ and I have referred to as smoothly decaying. This result is applied in Lemma \ref{waveequation}, where we construct $(\rho,\overline{J},\overline{E},\overline{B})$ satisfying Maxwell's equations and various other relations, just from the assumptions that $\rho$ satisfies the wave equation, has a restriction on the initial condition, and belongs to the Schwartz class.\\
\indent The second part of the paper is mainly concerned with deriving these relations, and proving a converse in the context of special relativity, that these relations are characterised by a no radiation condition in all inertial frames. In Lemma \ref{solution}, we characterise the smoothly decaying solutions of the wave equation $\square^{2}(\overline{E})=\overline{0}$ for electric fields, as well as the smoothly decaying electromagnetic solutions $(\overline{E},\overline{B})$ to Maxwell's equations in free space. This requires a careful Fourier analysis, and, in particular, we require the smoothly decaying hypothesis to apply the inversion theorem. We conclude the Lemma by proving that if $\square^{2}(\overline{E})=\overline{0}$, then we can find a free space solution $(\overline{E}_{0},\overline{B}_{0})$, such that $\bigtriangledown\times(\overline{E}-\overline{E}_{0})=\overline{0}$. This result is required in Lemma \ref{waveequation}.\\
\indent In Lemma \ref{electricmagnetic}, we characterise $\square^{2}(\overline{E})$ and $\square^{2}(\overline{B})$ in terms of the quantities $(\rho,\overline{J})$. Combining these last two Lemmas, we obtain the result of Lemma \ref{waveequation} mentioned above, that we can obtain a number of conditions on $(\rho,\overline{J},\overline{E},\overline{B})$ from essentially the assumption that $\rho$ satisfies the wave equation. In particularly, we can arrange for the magnetic field $\overline{B}$ and Poynting vector $\overline{E}\times\overline{B}$ to be both zero. In Lemma \ref{movingframes}, we strengthen this result, to show that if we can obtain the relations of Lemma \ref{waveequation} for $(\rho,\overline{J},\overline{E},\overline{B})$ in the rest frame $S$, then for \emph{any} inertial frame $S'$, we can extend the transformed charge and current $(\rho',\overline{J}')$ to $(\rho',\overline{J}',\overline{E}',\overline{B}')$, satisfying the same relations. This requires the transformation rules for quantities in $S$, and we rely on the fact that the transformed quantities are bounded, which follows from the smoothly decaying hypothesis, to apply Liouville's theorem.\\
\indent These last two Lemmas are the basis for the non-radiating condition, which I formulate in Definition \ref{nonradiating}, that in \emph{any} inertial frame $S'$ we can arrange for a solution $(\rho',\overline{J}',\overline{E}',\overline{B}')$, extending the transformed quantities $(\rho',\overline{J}')$, with $\overline{B}'=\overline{0}$. This certainly ensures that the transformed system doesn't radiate, according to the definition in \cite{G}, as the Poynting vector $\overline{E}'\times \overline{B}'$ is zero. Interestingly, in Lemma \ref{wavenonradiating}, a converse is proved, namely that for any system satisfying the non-radiating condition, in the rest frame $S$, $\rho$ and $\overline{J}$ satisfy the wave equation, and we can obtain all the relations proved in Lemma \ref{waveequation}. This suggests that the no-radiation condition might be enough to consider using wave equations for $(\rho,\overline{J})$ in atomic systems, which should not radiate according to Rutherford's observation.\\
\indent However, there are weaker definitions of non-radiation, which we give in Definition \ref{nonradiating2}. Some of these we are able to exclude, while supporting our hypothesis, which we do in Lemmas \ref{nononradiating2} and \ref{nononradiating3}, and, others exclusions we leave as conjectures. Certainly, a successful proof of these conjectures, which we believe might be possible with a careful analysis of Jefimenko's equations, would be a compelling argument for the use of the wave equation in quantum theory, and we should make some comparisons. Although Schrodinger's equation has a relativistic formulation in quantum field theory, through the use of the Klein-Gordon equation or the Dirac equation, it still lacks the relativistic invariance of the relations which we were able to derive, and, indeed, uses the possibility of radiating atomic systems to account for the Lamb shift in the hydrogen spectrum. We believe this shift might be accounted for in the discussion of the final part of the paper, and still believe that radiative behaviour would be a strange property. For example, the radiated energy would have to oscillate, if the system is to remain stable, and we do not observe this behaviour in the Lamb shift. A successful resolution of these issues is clearly required and we believe that we have made a step in the right direction.\\
\indent In the last part of this paper, we deal with a common criticism of classical electromagnetism, in it's failure to explain quantised phenomena, in particular the behaviour of the Balmer series or the results of the Franck-Hertz experiment. In Lemma \ref{waveequationlinks}, we start with the relations of Lemma \ref{waveequation}, and characterise these relations in terms of the coefficients governing the Fourier expansions of $\rho$ and $\overline{J}$, a coefficient relation which I later refer to as a radial transform condition. In Lemma \ref{eigenvalues}, I show that these relations can be obtained while also imposing the condition that $\overline{J}|_{S(r_{0})}=\overline{0}$, a natural requirement for an atomic system. This introduces a discreteness phenomena in terms of the zeroes of Bessel functions vanishing at $r_{0}$ and, in Lemma \ref{sphereorthogonal}, I prove a technical result on how these functions, combined with the spherical harmonics, provide an orthogonal basis for smooth functions on the ball $B(r_{0})$, vanishing at the boundary.\\
\indent In Lemma \ref{balmer}, we compute the energy stored in the electromagnetic field, when restricted to $B(r_{0})$, and find, that when the total charge confined to the ball is non-zero, we are able to remove the one continuous degree of freedom, to obtain quantisation of the energy values, in line with the Balmer series, see \cite{NR}. The ionisation condition, on the total charge $Q$ confined to the ball, seems quite interesting, and is observed in the Frank-Hertz experiment. Again, some comparative remarks should be made, namely the success of Schrodinger's equation, together with Bohr's atomic model, in accounting for the Balmer series. We were not able to verify the value of the Rydberg constant in this paper, due to computational issues, but we emphasise that we were able to get the ${1\over n^{2}}$ dependence, where $n$ is an integer. Possible, further computations of Bessel functions and spherical harmonics can be done to resolve this issue, and, again, we believe that we have gone some way to answering this potential criticism of the use of the wave equation. I believe there is some scope for developing the classical theory here, without relying on Planck's photonic theory of radiation, in how the system switches between electromagnetic energy levels, and in how light of certain energies, which we propose would come from the differences in these energy levels, is emitted.\\
\indent We make some observations which are not part of this paper, but instead direct the reader to further work. A principle justification for using the Schrodinger equation is its reliance on wave-particle duality. This requires that an electron can simultaneously be both a particle and a guiding wave, modelled on the square of the wave function $|\Psi|^{2}$. The particle nature of the electron can be observed in experiments such as Millikan's deflection or Compton scattering, while the wave nature, can be seen in the interference patterns of electron diffraction, see \cite{NR}. However, it seems that a dual nature which explains these experiments is a failure to find a better explanation, with the cost of developing a rather paradoxical theory. Instead, a wave function, modelling the behaviour of charge and current, and linking directly with classical electromagnetism, has been proposed here.\\
\indent The author believes that, using nonstandard analysis, one can develop the idea of the wave being composed of individual infinitesimal entities, which behave as particles, and propagate independently and essentially randomly throughout the wave. Some successes for this viewpoint can be found in \cite{heat}, where the heat equation solution is modelled as a random diffusion, and \cite{wave}, where the nonstandard wave equation is shown to perfectly approximate the wave equation at standard values. Instead of considering individual velocities and momenta, one can instead compute the \emph{distribution} of these quantities, and some unfinished progress in this direction has been made for the diffusion equation in \cite{heat2}. A similar idea is present in Boltzmann's derivation of the distribution of molecular speeds for ideal gases. In the case of the wave equation satisfying additional conditions which we have examined, we expect a sharp peak in the distribution function, as would be expected from a single particle. Similar consideration might apply to photons.\\
\indent We finally mention the theory of black body radiation, developed by Planck, and explained concisely in \cite{DV}. A criticism of classical electromagnetism has been its failure to explain the spectrum of black body radiation, by postulating that the radiated energy is independent of frequency, often referred to as the ultraviolet catastrophe. A central concept used in the experimentally successful resolution of this problem, was Planck's use of the quantisation of particular energies of molecular systems. The author hopes that the present paper goes some way to assuage these criticisms, without abandoning the main component of Planck's argument in quantum statistical mechanics.\\

\begin{section}{Jefimenko's Equations}

\indent We begin this paper, then, by first clarifying a technical issue surrounding Jefimenko's equations.

\begin{lemma}
\label{lorentz}

Suppose that $\{\overline{E},\overline{B}\}$ satisfy Maxwell equations, for given charge and current $\{\rho,\overline{J}\}$, and the potentials $\{\overline{A},V\}$, with;\\

$\overline{E}=-\bigtriangledown(V)-{\partial\overline{A}\over \partial t}$\\

$\overline{B}=\bigtriangledown\times\overline{A}$\\

are chosen to satisfy the Lorentz gauge condition;\\

$\bigtriangledown\centerdot\overline{A}+\mu_{0}\epsilon_{0}{\partial V\over \partial t}=0$\\

Suppose that $\{\overline{A}',V'\}$ also satisfy the Lorentz gauge condition, with the additional equations;\\

$\bigtriangledown^{2}(V')-\mu_{0}\epsilon_{0}{\partial^{2} V'\over \partial t^{2}}=-{\rho\over\epsilon_{0}}$ $(i)$\\

$\bigtriangledown^{2}(\overline{A}')-\mu_{0}\epsilon_{0}{\partial^{2} \overline{A}'\over \partial t^{2}}=-\mu_{0}\overline{J}$ $(ii)$\\

Then the corresponding fields $\{\overline{E}',\overline{B}'\}$ still satisfy Maxwell's equations for the same charge and current $\{\rho,\overline{J}\}$.

\end{lemma}

\begin{proof}
It follows from \cite{G} that $\{\overline{A},\overline{V}\}$ satisfy the equations $(i),(ii)$ from the lemma. Let $V''=V-V'$, $\overline{A}''=\overline{A}-\overline{A}'$, then $\{\overline{A}'',V''\}$ satisfy the equations;\\

$\bigtriangledown^{2}(V'')-\mu_{0}\epsilon_{0}{\partial^{2} V''\over \partial t^{2}}=0$ $(i)'$\\

$\bigtriangledown^{2}(\overline{A}'')-\mu_{0}\epsilon_{0}{\partial^{2} \overline{A}''\over \partial t^{2}}=\overline{0}$ $(ii)'$\\

and the Lorentz gauge condition. Let $\{\overline{E}'',\overline{B}''\}$ be the corresponding fields. Then, it is sufficient to show that they satisfy Maxwell's equation in vacuum;\\

$\bigtriangledown\centerdot\overline{E}''=0$ $(a)$\\

$\bigtriangledown\centerdot\overline{B}''=0$ $(b)$\\

$\bigtriangledown\times \overline{E}''=-{\partial \overline{B}''\over \partial t}$ $(c)$\\

$\bigtriangledown\times \overline{B}''=\mu_{0}\epsilon_{0}{\partial \overline{E}''\over \partial t}$ $(d)$\\

For $(a)$, we have, using the definition of $\overline{E}''$, the Lorentz gauge condition, and condition $(i)'$ that;\\

$\bigtriangledown\centerdot\overline{E}''=\bigtriangledown\centerdot(-\bigtriangledown(V'')-{\partial \overline{A''}\over \partial t})$\\

$=-\bigtriangledown^{2}(V'')-{\partial(\bigtriangledown\centerdot\overline{A''})\over \partial t}$\\

$=-\mu_{0}\epsilon_{0}{\partial^{2}V''\over \partial t^{2}}-{\partial (\bigtriangledown.\centerdot\overline{A''})\over \partial t}$\\

$=-\mu_{0}\epsilon_{0}{\partial^{2}V''\over \partial t^{2}}+\mu_{0}\epsilon_{0}{\partial^{2}V''\over \partial t^{2}}=0$\\

For $(b)$, we have, using vector analysis, see \cite{BK}, (4.46);\\

$\bigtriangledown\centerdot\overline{B}''=\bigtriangledown\centerdot(\bigtriangledown \times\overline{A''})=0$\\

For $(c)$, we have, using the definition of $\overline{E}''$, $\overline{B}''$ and vector analysis, see \cite{BK}, (4.47), that;\\

$\bigtriangledown\times \overline{E}''=\bigtriangledown\times(-\bigtriangledown(V'')-{\partial \overline{A''}\over \partial t})$\\

$=-(\bigtriangledown \times(\bigtriangledown (V))-{\partial (\bigtriangledown\times\overline{A}'')\over\partial t}$\\

$=-{\partial \overline{B}''\over \partial t}$\\

For $(d)$, we have, using the definition of $\overline{B}''$ and condition $(ii)'$ that;\\

$\bigtriangledown\times \overline{B}''=\bigtriangledown\times(\bigtriangledown\times\overline{A}'')$\\

A simple calculation shows that;\\

$\bigtriangledown\times(\bigtriangledown\times\overline{A}'')=(d_{1},d_{2},d_{3})=(c_{1},c_{2},c_{3})$\\

where;\\

$d_{1}={\partial^{2}a_{2}\over \partial x \partial y}+{\partial^{2}a_{3}\over \partial x \partial z}+(-{\partial^{2}a_{1}\over \partial y^{2}}-{\partial^{2}a_{1}\over \partial z^{2}})$\\

$d_{2}={\partial^{2}a_{1}\over \partial x \partial y}+{\partial^{2}a_{3}\over \partial y \partial z}+(-{\partial^{2}a_{2}\over \partial x^{2}}-{\partial^{2}a_{2}\over \partial x^{2}})$\\

$d_{3}={\partial^{2}a_{1}\over \partial x \partial z}+{\partial^{2}a_{2}\over \partial y \partial z}+(-{\partial^{2}a_{3}\over \partial x^{2}}-{\partial^{2}a_{3}\over \partial y^{2}})$\\

$c_{1}={\partial^{2}a_{2}\over \partial x \partial y}+{\partial^{2}a_{3}\over \partial x \partial z}+({\partial^{2}a_{1}\over \partial x^{2}}-\mu_{0}\epsilon_{0}{\partial^{2}a_{1}\over \partial t^{2} })$\\

$c_{2}={\partial^{2}a_{1}\over \partial x \partial y}+{\partial^{2}a_{3}\over \partial y \partial z}+({\partial^{2}a_{2}\over \partial x^{2}}-\mu_{0}\epsilon_{0}{\partial^{2}a_{2}\over \partial t^{2} })$\\

$c_{3}={\partial^{2}a_{1}\over \partial x \partial z}+{\partial^{2}a_{2}\over \partial y \partial z}+({\partial^{2}a_{3}\over \partial z^{2}}-\mu_{0}\epsilon_{0}{\partial^{2}a_{3}\over \partial t^{2}})$\\

whereas, using the definition of $\overline{E}''$;\\

$\mu_{0}\epsilon_{0}{\partial \overline{E}''\over \partial t}$\\

$=\mu_{0}\epsilon_{0}{\partial (-\bigtriangledown (V'')-{\partial \overline{A}'' \over \partial t})\over \partial t}$\\

$=-\mu_{0}\epsilon_{0}\bigtriangledown ({\partial V''\over\partial t})-\mu_{0}\epsilon_{0}{\partial^{2}\overline{A}''\over \partial t^{2}}$\\

It is, therefore, sufficient to prove that;\\

$-\mu_{0}\epsilon_{0}\bigtriangledown({\partial V''\over \partial t})=(e_{1},e_{2},e_{3})$, where;\\

$e_{1}={\partial^{2}a_{2}\over \partial x \partial y}+{\partial^{2}a_{3}\over \partial x \partial z}+{\partial^{2}a_{1}\over \partial x^{2}}$\\

$e_{2}={\partial^{2}a_{1}\over \partial x \partial y}+{\partial^{2}a_{3}\over \partial y \partial z}+{\partial^{2}a_{2}\over \partial x^{2}}$\\

$e_{3}={\partial^{2}a_{1}\over \partial x \partial z}+{\partial^{2}a_{2}\over \partial y \partial z}+{\partial^{2}a_{3}\over \partial z^{2}}$\\

This holds, using the Lorentz gauge condition, as;\\

$-\mu_{0}\epsilon_{0}\bigtriangledown({\partial V''\over \partial t})$\\

$=-\mu_{0}\epsilon_{0}\bigtriangledown (-{\bigtriangledown\centerdot \overline{A}''\over \mu_{0}\epsilon_{0} })$\\

$=\bigtriangledown (\bigtriangledown\centerdot \overline{A}'')$\\

$=\bigtriangledown ({\partial a_{1}\over \partial x}+{\partial a_{2}\over \partial y}+{\partial a_{3}\over \partial z})$\\

$=(e_{1},e_{2},e_{3})$\\

\end{proof}

\begin{lemma}
\label{jefimenko}
Let the potentials $\{\overline{A}',V'\}$ be defined be given as retarded potentials;\\

$V'(\overline{r},t)={1\over 4\pi\epsilon_{0}}\int {\rho(\overline{r}',t_{r})\over \mathfrak{r}}d\tau'$\\

$\overline{A}'(\overline{r},t)={\mu_{0}\over 4\pi}\int {\overline{J}(\overline{r}',t_{r})\over \mathfrak{r}}d\tau'$\\

then, assuming $\{\rho, \overline{J}\}$ satisfy the continuity equation and $\overline{J}$ vanishes at infinity, these potentials satisfy the Lorentz gauge condition and the equations in the hypotheses $(i)$ and $(ii)$ from Lemma \ref{lorentz}. In particular, the corresponding fields $\{\overline{E}',\overline{B}'\}$, given by Jefimenko's Equations;\\

$\overline{E}'(\overline{r},t)={1\over 4\pi\epsilon_{0}}\int [{\rho(\overline{r'},t_{r})\over \mathfrak{r}^{2}}\hat{\mathfrak{\overline{r}}}+{\dot{\rho}(\overline{r'},t_{r})\over c\mathfrak{r}}\hat{\mathfrak{\overline{r}}}-{\dot{\overline{J}}(\overline{r'},t_{r})\over c^{2}\mathfrak{r}}]d\tau'$\\

$\overline{B}'(\overline{r},t)={\mu_{0}\over 4\pi}\int [{\overline{J}(\overline{r'},t_{r})\over \mathfrak{r}^{2}}+{\dot{\overline{J}}(\overline{r'},t_{r})\over c\mathfrak{r}}]\times \hat{\mathfrak{\overline{r}}} d\tau'$\\

satisfy Maxwell's equations.

\end{lemma}

\begin{proof}
The first part of the claim, under the stated hypotheses, is proved in \cite{G}, p424, see also footnote 2 of that page and Exercise 10.8, with the solutions given in \cite{Gr}. For the final part of the claim, one can assume the existence of a solution $\{\rho,\overline{J},\overline{E},\overline{B}\}$ to Maxwell's equations, then construct potentials $\{V,\overline{A}\}$ abstractly, satisfying the Lorentz gauge condition, as is done in \cite{G}. Applying the result of Lemma \ref{lorentz}, we obtain the result. Alternatively, one can verify Maxwell's equations directly, for $\{\overline{E}',\overline{B}'\}$, using the method of Lemma \ref{lorentz}, just replacing the conditions $(i)'(ii)'$ on $\{V',\overline{A}'\}$, in the proof, with their non-homogeneous versions. The fact that the fields $\{\overline{E}',\overline{B}'\}$ are given by Jefimenko's equations is proved in \cite{G}, p427-428.

\end{proof}

\begin{rmk}
\label{jefimenko2}
There is an alternative strategy to construct explicit solutions for $(\overline{A},V)$, given the corresponding fields $\{\overline{E},\overline{B}\}$, satisfying Maxwell's equations, and the potentials satisfying the Lorentz gauge condition. Namely, one can use the explicit formulas, suitably rescaled, for solutions to the homogeneous and inhomogeneous wave equations given in \cite{E}, (p73 formula (22) and p82 formula (44)) respectively, with the initial conditions given by $\{V_{0},V_{0,t},\overline{A}_{0},\overline{A}_{0,t}\}$ and the driving terms given by $\{-{\rho\over\epsilon_{0}},-\mu_{0}\overline{J}\}$. In the context of Lemma \ref{jefimenko}, one can easily compute the initial conditions, and, in principle derive new formulas for $(\overline{A},V)$, replacing the retarded potentials, and for $\{\overline{E},\overline{B}\}$, replacing Jefimenko's equations.

\end{rmk}
\end{section}
\begin{section}{A No Radiation Condition}
We now clarify some results about smoothly decaying solutions to Maxwell's equations in free space.\\

\begin{defn}
\label{squares}
By Maxwell's equations in free space, we mean;\\

$(i)$. $div(\overline{E})=0$\\

$(ii)$. $div(\overline{B})=0$\\

$(iii)$. $\bigtriangledown\times\overline{E}=-{\partial \overline{B}\over \partial t}$\\

$(iv)$.  $\bigtriangledown\times\overline{B}=\mu_{0}\epsilon_{0}{\partial \overline{E}\over \partial t}$\\

We abbreviate the operator $\bigtriangledown^{2}-\mu_{0}\epsilon_{0}{\partial^{2}\over \partial t^{2}}$ by $\square^{2}$\\

\end{defn}

\begin{lemma}
\label{solution}
The smoothly decaying solutions of the wave equation $\square^{2}(\overline{E})=\overline{0}$ are given by;\\

$\overline{E}(\overline{x},t)=\int_{\mathcal{R}^{3}}\overline{A}(\overline{k})e^{i(\overline{k}\centerdot\overline{x}-\omega(\overline{k})t)}d\overline{k}+\int_{\mathcal{R}^{3}}\overline{B}(\overline{k})e^{i(\overline{k}\centerdot\overline{x}+\omega(\overline{k})t)}d\overline{k}$\\

where $A,B\subset S(\mathcal{R}^{3})$\\

while the smoothly decaying solutions of Maxwell's equations in free space, are given by;\\

$\overline{E}(\overline{x},t)=\int_{\overline{k}\in\mathcal{R}^{3}}\int_{S_{\overline{k}}}G(\overline{k},\overline{n})e^{i(\overline{k}\centerdot\overline{x}-\omega(\overline{k})t)}d\overline{S}_{\overline{k}}(\overline{n})d\overline{k}$\\

$+\int_{\overline{k}\in\mathcal{R}^{3}}\int_{S_{\overline{k}}}H(\overline{k},\overline{n})e^{i(\overline{k}\centerdot\overline{x}+\omega(\overline{k})t)}d\overline{S}_{\overline{k}}(\overline{n})d\overline{k}$ $(*)$\\

where $k=|\overline{k}|$, $\omega(\overline{k})=c|\overline{k}|={k\over \sqrt{\mu_{0}\epsilon_{0}}}$ and  $\{G,H\}\subset \mathcal{S}(M)$ where $S_{\overline{k}}=(S^{2}(\overline{k},1)\cap P_{\overline{k}})$, $P_{\overline{k}}=\{\overline{n}:(\overline{n}-\overline{k})\centerdot\overline{k}=0\}$, $d\overline{S}_{\overline{k}}(\overline{n})=(\overline{n}-\overline{k})dS_{\overline{k}}$ $M=\{(\overline{k},\overline{n})\in\mathcal{R}^{6}:(\overline{n}-\overline{k})\centerdot\overline{k}=0,|\overline{n}-\overline{k}|=1\}$ and $\mathcal{S}(M)=\{f\in C(M):\int_{S_{\overline{k}}}fd\overline{S}_{\overline{k}}\in\mathcal{S}(\mathcal{R}^{3},\mathcal{R}_{\geq 0},\mathcal{C})\}$\\

$\overline{B}(\overline{x},t)=\int_{\overline{k}\in\mathcal{R}^{3}}\int_{S_{\overline{k}}}\overline{M}(\overline{k},\overline{n})e^{i(\overline{k}\centerdot\overline{x}-\omega(\overline{k})t)}dS_{\overline{k}}(\overline{n})d\overline{k}$\\

$+\int_{\overline{k}\in\mathcal{R}^{3}}\int_{S_{\overline{k}}}\overline{N}(\overline{k},\overline{n})e^{i(\overline{k}\centerdot\overline{x}+\omega(\overline{k})t)}dS_{\overline{k}}(\overline{n})d\overline{k}$ $(**)$\\

where $k=|\overline{k}|$, $\omega(\overline{k})=c|\overline{k}|={k\over \sqrt{\mu_{0}\epsilon_{0}}}$,  $\overline{M}(\overline{k},\overline{n})={-G(\overline{k},\overline{n})\over\omega(\overline{k})}(\overline{k}\times \overline{n})$\\

$\overline{N}(\overline{k},\overline{n})={H(\overline{k},\overline{n})\over\omega(\overline{k})}(\overline{k}\times \overline{n})$\\

Finally, if $\overline{E}$ is a smoothly decaying solutions of the wave equation $\square^{2}(\overline{E})=\overline{0}$, there exists a pair $(\overline{E}_{0},\overline{B}_{0})$ which is a smoothly decaying solution of Maxwell's equations in free space, such that;\\

$\bigtriangledown\times(\overline{E}-\overline{E}_{0})=\overline{0}$\\

\end{lemma}

\begin{proof}
It is easily checked that the solutions $(*)$, $(**)$ satisfy $(i)-(iv)$ of Definition \ref{squares}.
Conversely, let $\{\overline{E},\overline{B}\}$ be smooth solutions of $(i)-(iv)$. We have, using $(i),(iii),(iv)$, that;\\

$(\bigtriangledown\times\overline{E})=-{\partial \overline{B}\over \partial t}$, and, hence;\\

$\bigtriangledown\times(\bigtriangledown\times\overline{E})=-(\bigtriangledown\times {\partial \overline{B}\over \partial t})$\\

$-\bigtriangledown^{2}\overline{E}=-{\partial\over \partial t}(\bigtriangledown\times\overline{B})=-\mu_{0}\epsilon_{0}{\partial^{2}\overline{E}\over \partial t^{2}}$\\

$\bigtriangledown^{2}\overline{E}={1\over c^{2}}{\partial^{2}\overline{E}\over \partial t^{2}}$ ($c={1\over \sqrt{\mu_{0}\epsilon_{0}}}$)\\

If $E_{i}\in \mathcal{S}_{(\mathcal{R}^{3},\mathcal{R})}(\mathcal{R}^{4})$ solve the wave equation, $\bigtriangledown^{2}E_{i}={1\over c^{2}}{\partial^{2}E_{i}\over \partial t^{2}}$, where $\mathcal{S}_{(\mathcal{R}^{3},\mathcal{R})}(\mathcal{R}^{4})=\{f\in C^{\infty}(\mathcal{R}^{4}):f_{t}\in\mathcal{S}(\mathcal{R}^{3})\}$, for $t\in\mathcal{R}$, then;\\

$E_{i}(\overline{x},t)=({1\over 2\pi})^{3}\int_{\mathcal{R}^{3}}\hat{E_{i}}(\overline{k},t)e^{i\overline{k}\centerdot\overline{x}}d\overline{k}$\\

by the inversion theorem of Fourier analysis, see \cite{E} and \cite{SS}. Hence;\\

$\bigtriangledown^{2}E_{i}=-({1\over 2\pi})^{3}\int_{\mathcal{R}^{3}}|\overline{k}|^{2}\hat{E_{i}}(\overline{k},t)e^{i\overline{k}\centerdot\overline{x}}d\overline{k}$\\

${\partial^{2}E_{i}\over \partial t^{2}}=({1\over 2\pi})^{3}\int_{\mathcal{R}^{3}}{\partial^{2}\hat{E_{i}}\over \partial t^{2}}(\overline{k},t)e^{i\overline{k}\centerdot\overline{x}}d\overline{k}$\\

$\bigtriangledown^{2}E_{i}-{1\over c^{2}}{\partial^{2}E_{i}\over \partial t^{2}}=({1\over 2\pi})^{3}\int_{\mathcal{R}^{3}}(-|\overline{k}|^{2}\hat{E_{i}}-{1\over c^{2}}{\partial^{2}\hat{E_{i}}\over \partial t^{2}})(\overline{k},t)e^{i\overline{k}\centerdot\overline{x}}d\overline{k}=0$\\

so that;\\

$|\overline{k}|^{2}\hat{E_{i}}+{1\over c^{2}}{\partial^{2}\hat{E_{i}}\over \partial t^{2}}=0$\\

using the inversion formula again.\\

$\hat{E_{i}}(\overline{k},t)=A_{i}(\overline{k})e^{-i|\overline{k}|ct}+B_{i}(\overline{k})e^{i|\overline{k}|ct}$, as $\hat{E_{i}}(\overline{k},t)\in\mathcal{S}_{(\mathcal{R}^{3},\mathcal{R})}(\mathcal{R}^{4})$\\

$E_{i}(\overline{x},t)=({1\over 2\pi})^{3}(\int_{\mathcal{R}^{3}}(A_{i}(\overline{k})e^{-i|\overline{k}|ct})e^{i\overline{k}\centerdot\overline{x}}d\overline{k}+\int_{\mathcal{R}^{3}}(B_{i}(\overline{k})e^{i|\overline{k}|ct})e^{i\overline{k}\centerdot\overline{x}}d\overline{k})$\\

$=({1\over 2\pi})^{3}(\int_{\mathcal{R}^{3}}A_{i}(\overline{k})e^{i(\overline{k}\centerdot\overline{x}-\omega(\overline{k})t)}d\overline{k}+\int_{\mathcal{R}^{3}}B_{i}(\overline{k})e^{i(\overline{k}\centerdot\overline{x}+\omega(\overline{k})t)}d\overline{k})$\\

$\overline{E}(\overline{x},t)=\int_{\mathcal{R}^{3}}\overline{A}(\overline{k})e^{i(\overline{k}\centerdot\overline{x}-\omega(\overline{k})t)}d\overline{k}+\int_{\mathcal{R}^{3}}\overline{B}(\overline{k})e^{i(\overline{k}\centerdot\overline{x}+\omega(\overline{k})t)}d\overline{k}$\\

where $A,B\subset S(\mathcal{R}^{3})$, as required for the first part. Using $(i)$, we have that;\\

$\int_{\mathcal{R}^{3}}(\overline{A}(\overline{k})\centerdot i\overline{k})e^{i(\overline{k}\centerdot\overline{x}-\omega(\overline{k})t)}d\overline{k}+\int_{\mathcal{R}^{3}}(\overline{B}(\overline{k})\centerdot\overline{k})e^{i(\overline{k}\centerdot\overline{x}+\omega(\overline{k})t)}d\overline{k}=0$\\

At $t=0$ and $t=1$, we obtain that;\\

$\int_{\mathcal{R}^{3}}(\overline{A}(\overline{k})\centerdot i\overline{k})e^{i(\overline{k}\centerdot\overline{x})}d\overline{k}+\int_{\mathcal{R}^{3}}(\overline{B}(\overline{k})\centerdot\overline{k})e^{i(\overline{k}\centerdot\overline{x})}d\overline{k}=0$\\

$\int_{\mathcal{R}^{3}}(\overline{A}(\overline{k})\centerdot i\overline{k})e^{i(\overline{k}\centerdot\overline{x}-ck)}d\overline{k}+\int_{\mathcal{R}^{3}}(\overline{B}(\overline{k})\centerdot\overline{k})e^{i(\overline{k}\centerdot\overline{x}+ck)}d\overline{k}=0$\\

As this holds for all $\overline{x}\in\overline{R}^{3}$, using the inversion formula, we obtain that;\\

$(\overline{A}(\overline{k})+\overline{B}(\overline{k}))\centerdot \overline{k}=0$\\

$(\overline{A}(\overline{k})e^{-ck}+\overline{B}(\overline{k})e^{ck})\centerdot \overline{k}=0$\\

Assuming $k\neq 0$, we obtain that;\\

$\overline{A}(\overline{k})\centerdot \overline{k}=\overline{B}(\overline{k})\centerdot \overline{k}=0$\\

so that;\\

$A(\overline{k})=\int_{S_{\overline{k}}}G(\overline{k},\overline{n})d\overline{S}_{\overline{k}}(\overline{n})$, $B(\overline{k})=\int_{S_{\overline{k}}}H(\overline{k},\overline{n})d\overline{S}_{\overline{k}}(\overline{n})$ and the first part of the result $(*)$ follows.\\

Using the same argument, we obtain that;\\

$\overline{B}(\overline{x},t)$\\

$=\int_{\overline{k}\in\mathcal{R}^{3}}\int_{S_{\overline{k}}}\overline{K}(\overline{k},\overline{n})e^{i(\overline{k}\centerdot\overline{x}-\omega(\overline{k})t)}d\overline{S}_{\overline{k}}(\overline{n})d\overline{k}+\int_{\overline{k}\in\mathcal{R}^{3}}\int_{S_{\overline{k}}}\overline{L}(\overline{k},\overline{n})e^{i(\overline{k}\centerdot\overline{x}+\omega(\overline{k})t)}d\overline{S}_{\overline{k}}(\overline{n})d\overline{k}$ $(\dag\dag)$\\

Using $(iii)$;\\

${\partial \overline{B}\over \partial t}(\overline{x},t)$\\

$=\int_{\overline{k}\in\mathcal{R}^{3}}\int_{S_{\overline{k}}}iG(\overline{k},\overline{n})e^{i(\overline{k}\centerdot\overline{x}-\omega(\overline{k})t)}(d\overline{S}_{\overline{k}}(\overline{n})\times\overline{k})d\overline{k}$\\

$+\int_{\overline{k}\in\mathcal{R}^{3}}\int_{S_{\overline{k}}}iH(\overline{k},\overline{n})e^{i(\overline{k}\centerdot\overline{x}-\omega(\overline{k})t)}(d\overline{S}_{\overline{k}}(\overline{n})\times\overline{k})d\overline{k}$\\

$=\int_{\overline{k}\in\mathcal{R}^{3}}\int_{S_{\overline{k}}}iG(\overline{k},\overline{n})(({\overline{n}-\overline{k}})^{\hat{}}\times \overline{k})e^{i(\overline{k}\centerdot\overline{x}-\omega(\overline{k})t)}dS_{\overline{k}}(\overline{n})d\overline{k}$\\

$+\int_{\overline{k}\in\mathcal{R}^{3}}\int_{S_{\overline{k}}}iH(\overline{k},\overline{n})(({\overline{n}-\overline{k}})^{\hat{}}\times \overline{k})e^{i(\overline{k}\centerdot\overline{x}+\omega(\overline{k})t)}dS_{\overline{k}}(\overline{n})d\overline{k}$\\

$=\int_{\overline{k}\in\mathcal{R}^{3}}\int_{S_{\overline{k}}}iG(\overline{k},\overline{n})(\overline{n}\times \overline{k})e^{i(\overline{k}\centerdot\overline{x}-\omega(\overline{k})t)}dS_{\overline{k}}(\overline{n})d\overline{k}$\\

$+\int_{\overline{k}\in\mathcal{R}^{3}}\int_{S_{\overline{k}}}iH(\overline{k},\overline{n})(\overline{n}\times \overline{k})e^{i(\overline{k}\centerdot\overline{x}+\omega(\overline{k})t)}dS_{\overline{k}}(\overline{n})]d\overline{k}$\\

$\overline{B}(\overline{x},t)$\\

$=\int_{\overline{k}\in\mathcal{R}^{3}}\int_{S_{\overline{k}}}{-G(\overline{k},\overline{n})\over\omega(\overline{k})}(\overline{n}\times \overline{k})e^{i(\overline{k}\centerdot\overline{x}-\omega(\overline{k})t)}dS_{\overline{k}}(\overline{n})d\overline{k}$\\

$+\int_{\overline{k}\in\mathcal{R}^{3}}\int_{S_{\overline{k}}}{H(\overline{k},\overline{n})\over\omega(\overline{k})}(\overline{n}\times \overline{k})e^{i(\overline{k}\centerdot\overline{x}+\omega(\overline{k})t)}dS_{\overline{k}}(\overline{n})]d\overline{k}+\overline{\theta}(\overline{x})$\\

$=\int_{\overline{k}\in\mathcal{R}^{3}}\int_{S_{\overline{k}}}\overline{M}(\overline{k},\overline{n})e^{i(\overline{k}\centerdot\overline{x}-\omega(\overline{k})t)}dS_{\overline{k}}(\overline{n})d\overline{k}$\\

$+\int_{\overline{k}\in\mathcal{R}^{3}}\int_{S_{\overline{k}}}\overline{N}(\overline{k},\overline{n})e^{i(\overline{k}\centerdot\overline{x}+\omega(\overline{k})t)}dS_{\overline{k}}(\overline{n})d\overline{k}$ $(**)$\\

where $\overline{M}(\overline{k},\overline{n})={-G(\overline{k},\overline{n})\over\omega(\overline{k})}(\overline{k}\times \overline{n})$, $\overline{N}(\overline{k},\overline{n})={H(\overline{k},\overline{n})\over\omega(\overline{k})}(\overline{k}\times \overline{n})$, and $\overline{\theta}(\overline{x})=0$, using $(\dag\dag)$.\\

This proves the second part of the result. For the final part, we can use the first part to write;\\

$\overline{E}(\overline{x},t)=\int_{\mathcal{R}^{3}}\overline{A}(\overline{k})e^{i(\overline{k}\centerdot\overline{x}-\omega(\overline{k})t)}d\overline{k}+\int_{\mathcal{R}^{3}}\overline{B}(\overline{k})e^{i(\overline{k}\centerdot\overline{x}+\omega(\overline{k})t)}d\overline{k}$, $(\dag)$\\

where $\overline{A},\overline{B}\subset S(\mathcal{R}^{3})$. For $\overline{k}\neq\overline{0}$, let;\\

$\overline{A}_{2}(\overline{k})={(\overline{A}(\overline{k})\centerdot\overline{k})\over |\overline{k}|^{2}}\overline{k}=A_{3}(\overline{k})\overline{k}$\\

$\overline{A}_{1}(\overline{k})=\overline{A}(\overline{k})-\overline{A}_{2}(\overline{k})$\\

Then $\overline{A}_{1}(\overline{k})\centerdot\overline{k}=0$, and we can write;\\

$\overline{A}_{1}(\overline{k})=A_{2}(\overline{k})(\overline{n}-\overline{k})$ with $\overline{n}\in S_{\overline{k}}$ and $A_{2}\in S(\mathcal{R}^{3})$\\

so that;\\

$\overline{A}_{1}(\overline{k})=\int_{S_{\overline{k}}}G(\overline{k},\overline{n})d\overline{S_{\overline{k}}}(\overline{n})$ with $G\in S(M)$\\

It follows that we can write;\\

$\int_{\mathcal{R}^{3}}\overline{A}(\overline{k})e^{i(\overline{k}\centerdot\overline{x}-\omega(\overline{k})t)}d\overline{k}$\\

$=\int_{\overline{k}\in \mathcal{R}^{3}}\int_{S_{\overline{k}}}G(\overline{k},\overline{n})e^{i(\overline{k}\centerdot\overline{x}-\omega(\overline{k})t)}d\overline{S_{\overline{k}}}(\overline{n})d\overline{k}+\int_{\overline{k}\in \mathcal{R}^{3}}A_{3}(\overline{k})\overline{k}e^{i(\overline{k}\centerdot\overline{x}-\omega(\overline{k})t)}\overline{k}d\overline{k}$\\

Similarly, repeating the procedure for $\overline{B}(\overline{k})$, we can write;\\

$\int_{\mathcal{R}^{3}}\overline{B}(\overline{k})e^{i(\overline{k}\centerdot\overline{x}+\omega(\overline{k})t)}d\overline{k}$\\

$=\int_{\overline{k}\in \mathcal{R}^{3}}\int_{S_{\overline{k}}}H(\overline{k},\overline{n})e^{i(\overline{k}\centerdot\overline{x}+\omega(\overline{k})t)}d\overline{S_{\overline{k}}}(\overline{n})d\overline{k}+\int_{\overline{k}\in \mathcal{R}^{3}}B_{3}(\overline{k})\overline{k}e^{i(\overline{k}\centerdot\overline{x}+\omega(\overline{k})t)}\overline{k}d\overline{k}$\\

We set;\\

$\overline{E}_{0}(\overline{x},t)=\int_{\overline{k}\in \mathcal{R}^{3}}\int_{S_{\overline{k}}}G(\overline{k},\overline{n})e^{i(\overline{k}\centerdot\overline{x}-\omega(\overline{k})t)}d\overline{S_{\overline{k}}}(\overline{n})d\overline{k}+\int_{\overline{k}\in \mathcal{R}^{3}}\int_{S_{\overline{k}}}H(\overline{k},\overline{n})e^{i(\overline{k}\centerdot\overline{x}+\omega(\overline{k})t)}d\overline{S_{\overline{k}}}(\overline{n})d\overline{k}$\\

We need to check that if $\overline{E}$ is real, then so is $\overline{E}_{0}$. If $\overline{E}={\overline{E}}^{*}$, then, using $(\dag)$ and equating coefficients, we have that;\\

$\overline{A}(\overline{k})^{*}=\overline{B}(-\overline{k})$, $\overline{B}(\overline{k})^{*}=\overline{A}(-\overline{k})$\\

It follows that;\\

$\overline{A}_{2}(\overline{k})^{*}={(\overline{A}(\overline{k})^{*}\centerdot\overline{k})\over |\overline{k}|^{2}}\overline{k}$\\

$={(\overline{B}(-\overline{k})\centerdot\overline{k})\over |\overline{k}|^{2}}\overline{k}$\\

$=-{(\overline{B}(-\overline{k})\centerdot -\overline{k})\over |-\overline{k}|^{2}}\overline{k}$\\

$=\overline{B}_{2}(-\overline{k})$\\

and, similarly $\overline{B}_{2}(\overline{k})^{*}=\overline{A}_{2}(-\overline{k})$, so that $\overline{E}-\overline{E}_{0}$ is real, and, therefore, $\overline{E}_{0}$ is real.\\

Using the second part, we can find $\overline{B}_{0}$ such that $(\overline{E}_{0},\overline{B}_{0})$ is a smoothly decaying solution of Maxwell's equations in free space. Finally, we compute;\\

$\bigtriangledown\times (\overline{E}-\overline{E}_{0})$\\

$=\bigtriangledown\times (\int_{\overline{k}\in \mathcal{R}^{3}}A_{3}(\overline{k})e^{i(\overline{k}\centerdot\overline{x}-\omega(\overline{k})t)}\overline{k}d\overline{k}+\int_{\overline{k}\in \mathcal{R}^{3}}B_{3}(\overline{k})e^{i(\overline{k}\centerdot\overline{x}+\omega(\overline{k})t)}\overline{k}d\overline{k})$\\

$=\int_{\overline{k}\in \mathcal{R}^{3}}A_{3}(\overline{k})e^{i(\overline{k}\centerdot\overline{x}-\omega(\overline{k})t)}(-i\overline{k}\times\overline{k})d\overline{k}$\\

$+\int_{\overline{k}\in \mathcal{R}^{3}}B_{3}(\overline{k})e^{i(\overline{k}\centerdot\overline{x}+\omega(\overline{k})t)}(-i\overline{k}\times\overline{k})d\overline{k}=\overline{0}$\\

as required.\\

\end{proof}

\begin{lemma}
\label{electricmagnetic}
Let $(\rho,\overline{J},\overline{E},\overline{B})$ satisfy Maxwell's equations, then;\\

$\square^{2}(\overline{E})={\bigtriangledown(\rho)\over\epsilon_{0}}+\mu_{0}{\partial\overline{J}\over \partial t}$\\

$\square^{2}(\overline{B})=-\mu_{0}(\bigtriangledown\times\overline{J})$\\

\end{lemma}
\begin{proof}
As is done in \cite{G}, we can choose potentials $(V,\overline{A})$ such that;\\

$\overline{E}=-\bigtriangledown (V)-{\partial \overline{A}\over \partial t}$\\

$\overline{B}=\bigtriangledown\times \overline{A}$\\

and $(V,\overline{A})$ satisfy the Lorentz gauge condition. It follows, see Lemma \ref{lorentz}, that $(V,\overline{A})$ satisfy the equations;\\

$\square^{2}(V)=-{\rho\over\epsilon_{0}}$\\

$\square^{2}(\overline{A})=-\mu_{0}\overline{J}$\\

It is an easy exercise in vector calculus to show that $\square^{2}$ commutes with the gradient operator $\bigtriangledown$, the curl operator $\bigtriangledown\times$ and partial differentiation ${\partial\over \partial t}$. Therefore, we compute;\\

$\square^{2}(\overline{E})=\square^{2}(-\bigtriangledown (V)-{\partial \overline{A}\over \partial t})$\\

$=-\bigtriangledown(\square^{2}(V))-{\partial (\square^{2}(\overline{A}))\over \partial t}$\\

$=-\bigtriangledown(-{\rho\over\epsilon_{0}})-{\partial (-\mu_{0}\overline{J}) \over \partial t}$\\

$=\bigtriangledown({\rho\over\epsilon_{0}})+\mu_{0}{\partial\overline{J} \over \partial t}$\\

and;\\

$\square^{2}(\overline{B})=\square^{2}(\bigtriangledown\times \overline{A})$\\

$=\bigtriangledown\times \square^{2}(\overline{A})$\\

$=\bigtriangledown\times (-\mu_{0}\overline{J})$\\

$=-\mu_{0}(\bigtriangledown\times \overline{J})$\\

\end{proof}

\begin{lemma}
\label{waveequation}
Let $\rho$ satisfy the wave equation $\square^{2}(\rho)=0$, with the initial conditions ${\partial \rho\over \partial t}|_{t=0}=0$ and $\rho|_{t=0}\in S(\overline{R}^{3})$, then there exists $\overline{J}$ such that $(\rho,\overline{J})$ satisfies the continuity equation, $\overline{J}$ satisfies the wave equation $\square^{2}(\overline{J})=\overline{0}$ with $\bigtriangledown\times\overline{J}=\overline{0}$, and $(\overline{E},\overline{B})$ such that $(\rho,\overline{J},\overline{E},\overline{B})$ satisfy Maxwell's equations, with $\square^{2}(\overline{E})=\overline{0}$ and $\overline{B}=\overline{0}$. In particular, the Poynting vector $\overline{E}\times \overline{B}=\overline{0}$\\
\end{lemma}

\begin{proof}
Define $\overline{J}$ by;\\

$\overline{J}={-1\over \epsilon_{0}\mu_{0}}\int_{0}^{t}\bigtriangledown(\rho)ds$\\

Then, by the fundamental theorem of calculus, we have ${\partial \overline{J}\over\partial t}=-{\bigtriangledown(\rho)\over \epsilon_{0}\mu_{0}}$. In particular, ${\bigtriangledown(\rho)\over \epsilon_{0}}+\mu_{0}{\partial \overline{J}\over\partial t}=\overline{0}$, $(*)$. Applying the divergence operator, differentiating under the integral sign, using the wave equation for $\rho$, and the first initial condition, we have;\\

$div(\overline{J})={-1\over \epsilon_{0}\mu_{0}}\int_{0}^{t}div(\bigtriangledown(\rho))ds$\\

$={-1\over \epsilon_{0}\mu_{0}}\int_{0}^{t}\bigtriangledown^{2}(\rho))ds$\\

$=-\int_{0}^{t}{\partial^{2}\rho\over\partial s^{2}}ds$\\

$=-{\partial \rho\over \partial t}$\\

so that $(\rho,\overline{J})$ satisfy the continuity equation. By the second initial condition, that $\rho_{0}$ belongs to the Schwartz class, and the general solution of the homogeneous wave equation with initial conditions, see \cite{E}, we have that $\overline{J}$ vanishes at infinity. It follows, applying the result of Lemma \ref{jefimenko}, that we can find a pair $(\overline{E},\overline{B})$, such that $(\rho,\overline{J},\overline{E},\overline{B})$ satisfy Maxwell's equations, and, by the result of Lemma \ref{electricmagnetic} and the condition $(*)$, we have that $\square^{2}(\overline{E})=\overline{0}$, so that $\overline{E}$ also satisfies the wave equation. Moreover, by the explicit formulas in Jefimenko's equations, one can check that $\overline{E}$ is also smoothly decaying. Then, using the result of Lemma \ref{solution}, we can find $(\overline{E}_{0},\overline{B}_{0})$, which are smoothly decaying solutions of Maxwell's equations in free space and with $\bigtriangledown\times(\overline{E}-\overline{E}_{0})=\overline{0}$, $(\dag)$. Clearly, $(\rho,\overline{J},\overline{E}-\overline{E}_{0},\overline{B}-\overline{B}_{0})$ still satisfy Maxwell's equations, and, by $(\dag)$ and $(ii)$, we have that;\\

$-{\partial (\overline{B}-\overline{B}_{0})\over \partial t}$\\

$=\bigtriangledown\times(\overline{E}-\overline{E}_{0})=\overline{0}$\\

that is the field $\overline{B}-\overline{B}_{0}$ is magnetostatic. Set $\mu_{0}\overline{J}_{0}=\bigtriangledown\times (\overline{B}-\overline{B}_{0})$, then $(0,\overline{J}_{0},\overline{0},\overline{B}-\overline{B}_{0})$ satisfy Maxwell's equations, so that, subtracting solutions, $(\rho,\overline{J}-\overline{J}_{0},\overline{E}-\overline{E}_{0},\overline{0})$ also satisfies Maxwell's equations. We must have then that $(\rho,\overline{J}-\overline{J}_{0})$ satisfies the continuity equation. As $(\overline{E}_{0},\overline{B}_{0})$ were solutions to Maxwell's equations in free space, we have $\square^{2}(\overline{E}_{0})=\overline{0}$, so that, as $\square^{2}(\overline{E})=\overline{0}$, we must have $\square^{2}(\overline{E}-\overline{E}_{0})=\overline{0}$ as well. By Lemma \ref{electricmagnetic}, we have that;\\

$\bigtriangledown\times (\overline{J}-\overline{J}_{0})={-1\over \mu_{0}}\square^{2}(\overline{0})=\overline{0}$, $(\dag)$\\

By elementary vector calculus, $(\dag)$, the continuity equation, the fact that $\square^{2}(\overline{E}-\overline{E}_{0})=\overline{0}$, and the first result of Lemma \ref{electricmagnetic}, we have that;\\

$\bigtriangledown^{2}(\overline{J}-\overline{J}_{0})$\\

$=\bigtriangledown(div(\overline{J}-\overline{J}_{0}))-\bigtriangledown\times (\bigtriangledown\times (\overline{J}-\overline{J}_{0}))$\\

$=\bigtriangledown(div(\overline{J}-\overline{J}_{0}))$\\

$=-\bigtriangledown({\partial \rho\over \partial t})$\\

$=-{\partial(-\epsilon_{0}\mu_{0}{\partial(\overline{J}-\overline{J}_{0})\over \partial t})\over \partial t}$\\

$={1\over c^{2}}{\partial^{2}(\overline{J}-\overline{J}_{0})\over \partial t^{2}}$\\

so that $\square^{2}(\overline{J}-\overline{J}_{0})=\overline{0}$. This proves the main claim. The fact that the Poynting vector is zero follows trivially from the fact that the magnetic field vanishes.

\end{proof}

We now strengthen this result;\\

\begin{lemma}
\label{movingframes}
Let $(\rho,\overline{J},\overline{E},\overline{B})$ satisfy the conclusions of Lemma \ref{waveequation}, then in any inertial frame $S'$ moving with velocity vector $\overline{v}$ relative to $S$, if $(\rho',\overline{J}')$ are the transformed charge distribution and current, there exists a pair $(\overline{E}',\overline{B}')$ such that $(\rho',\overline{J}',\overline{E}',\overline{B}')$ satisfy Maxwell's equations in $S'$, and $\square'^{2}(\overline{E}')=\overline{0}$ with $\overline{B}'=\overline{0}$ in $S'$. In particular, the Poynting vector $\overline{E}'\times\overline{B}'=\overline{0}$. Moreover, the pair $(\rho',\overline{J}')$ still satisfy the wave equations $\square'^{2}(\rho')=0$ and $\square'^{2}(\overline{J}')=\overline{0}$, with $\bigtriangledown\times \overline{J}'=\overline{0}$.

\end{lemma}

\begin{proof}
Let $(\rho',\overline{J}',\overline{E}'',\overline{B}'')$ be the transformed quantities in $S'$, corresponding to $(\rho,\overline{J},\overline{E},\overline{B})$ in $S$. The transformation rule for the electric field, see \cite{lcl}, is given by;\\

$\overline{E}''=\overline{E}_{||}+\gamma(\overline{E}_{\perp}+\overline{v}\times\overline{B})$\\

where $||$ and $\perp$ denote the parallel and perpendicular components respectively. Note that $\overline{E}_{||}$ is defined in the equation by $(\overline{E}\centerdot\overline{v}){\overline{v}\over |\overline{v}|^{2}}$, and $\overline{E}_{\perp}$ by $\overline{E}-\overline{E}_{||}$. As $\overline{v}\times\overline{B}=\overline{0}$, from the assumption that $\overline{B}=\overline{0}$ in $S$, we have that an observer in $S'$
sees the electric field;\\

$\overline{E}''=\overline{E}_{||}+\gamma\overline{E}_{\perp}$\\

Let $\square'^{2}$ be the d'Alembertian operators in $S'$, then using the Lorentz invariance of the d'Alembertian operator, the obvious fact that it commutes with parallel and perpendicular components, the above transformation rule, and the fact that $\square^{2}(\overline{E})=\overline{0}$, we have;\\

$\square'^{2}(\overline{E}'')=\square'^{2}(\overline{E}_{||}+\gamma\overline{E}_{\perp})$\\

$=\square^{2}(\overline{E}_{||}+\gamma\overline{E}_{\perp})$\\

$=(\square^{2}(\overline{E}))_{||}+\gamma(\square^{2}\overline{E})_{\perp}=\overline{0}$\\

Similarly the transformation rule for the magnetic field, see \cite{lcl}, is given by;\\

$\overline{B}''=\overline{B}_{||}+\gamma(\overline{B}_{\perp}-{\overline{v}\times\overline{E}\over c^{2}})$\\

which, using the fact that $\overline{B}=\overline{0}$ in $S$ again, becomes;\\

$\overline{B}''=-{\gamma\over c^{2}}(\overline{v}\times\overline{E})$\\

Using, a similar argument to the above, this time using the fact that the d'Alembertian commutes with taking a cross product with $\overline{v}$, we have;\\

$\square'^{2}(\overline{B}'')=\square'^{2}(-{\gamma\over c^{2}}(\overline{v}\times\overline{E}))$\\

$=\square^{2}(-{\gamma\over c^{2}}(\overline{v}\times\overline{E}))$\\

$=-{\gamma\over c^{2}}(\overline{v}\times\square^{2}(\overline{E}))=\overline{0}$\\

As in Lemma \ref{waveequation}, and using the last part of the result of Lemma \ref{solution}, with the fact that $\square'^{2}(\overline{E}'')=\overline{0}$, we can find a pair $(\overline{E}_{0}'',\overline{B}_{0}'')$ which are smoothly decaying solutions of Maxwell's equation in free space, and with $\bigtriangledown\times (\overline{E}''-\overline{E}_{0}'')=\overline{0}$, $(\dag)$. Then clearly, we still have that;\\

$\square'^{2}(\overline{E}''-\overline{E}_{0}'')=\square'^{2}(\overline{B}''-\overline{B}_{0}'')=\overline{0}$\\

and, moreover, by Maxwell's equations in $S'$ and $(\dag)$;\\

$\bigtriangledown'\times (\overline{E}''-\overline{E}_{0}'')$\\

$=-{\partial\over \partial t}(\overline{B}''-\overline{B}_{0}'')=\overline{0}$\\

so that $\overline{B}''-\overline{B}_{0}''$ is magnetostatic. However, we then have;\\

$\square'^{2}(\overline{B}''-\overline{B}_{0}'')=\bigtriangledown'^{2}(\overline{B}''-\overline{B}_{0}'')=\overline{0}$\\

so that $\overline{B}''-\overline{B}_{0}''$ satisfies Laplace's equation and is harmonic. Using the fact that $\overline{B}''-\overline{B}_{0}''$ is bounded, we can use Liouville's theorem to conclude that $\overline{B}''-\overline{B}_{0}''$ is constant, and using the fact that $\overline{B}''-\overline{B}_{0}''$ vanishes at infinity, that in fact $\overline{B}''-\overline{B}_{0}''=\overline{0}$. Setting $\overline{E}'=\overline{E}''-\overline{E}_{0}''$ and $\overline{B}'=\overline{B}''-\overline{B}_{0}''$ then gives the firs result. The result about the Poynting vector is clear. Finally, we have the transformation rules for current and charge, see \cite{lcl}, given by;\\

$\rho'=\gamma(\rho-{vJ_{||}\over c^{2}})$\\

$\overline{J}'=\gamma(\overline{J}_{||}-\rho\overline{v})+\overline{J}_{\perp}$\\

where $\overline{v}=v\hat{\overline{v}}$ and $\overline{J}_{||}=J_{||}\hat{\overline{J}_{||}}$. We then compute, using the usual commutation rules, the transformation rules just given, and the fact that $\square^{2}(\rho)=0$ and $\square^{2}(\overline{J})=\overline{0}$ in $S$, that;\\

$\square'^{2}(\rho')=\square'^{2}(\gamma(\rho-{vJ_{||}\over c^{2}}))$\\

$=\square^{2}(\gamma(\rho-{vJ_{||}\over c^{2}}))$\\

$=\gamma(\square^{2}(\rho)-{v\square^{2}(J_{||})\over c^{2}})=0$\\

and;\\

$\square'^{2}(\overline{J}')=\square'^{2}(\gamma(\overline{J}_{||}-\rho\overline{v})+\overline{J}_{\perp})$\\

$=\square^{2}(\gamma(\overline{J}_{||}-\rho\overline{v})+\overline{J}_{\perp})$\\

$=(\gamma(\square^{2}(\overline{J})_{||}-\square^{2}(\rho)\overline{v})+\square^{2}(\overline{J})_{\perp})=\overline{0}$\\

Finally, the fact that $\bigtriangledown'\times \overline{J}'=\overline{0}$ follows from the result that $\square'^{2}(\overline{B}')=\overline{0}$ and the second result in Lemma \ref{electricmagnetic}.

\end{proof}

We now prove a kind of converse to this result. We first require a definition;\\

\begin{defn}
\label{nonradiating}
Let $(\rho,\overline{J})$ be a charge distribution and current, satisfying the continuity equation in the rest frame $S$. Then, we say that $(\rho,\overline{J})$ is non-radiating if in any inertial frame $S'$, with velocity vector $\overline{v}$, for the transformed current and charge $(\rho',\overline{J}')$, there exist electric and magnetic fields $(\overline{E}',\overline{B}')$ in $S'$ such that $(\rho',\overline{J}',\overline{E}',\overline{B}')$ satisfy Maxwell's equations in $S'$ and with $\overline{B}'=\overline{0}$.
\end{defn}

\begin{lemma}
\label{wavenonradiating}
Let $(\rho,\overline{J})$, as in Definition \ref{nonradiating}, be non-radiating, then $((\rho,\overline{J}))$ satisfy the wave equations $\square^{2}(\rho)=0$ and $\square^{2}(\overline{J})=\overline{0}$.

\end{lemma}

\begin{proof}
By the definition of non-radiating, there exist fields $(\overline{E},\overline{B})$ in the rest frame $S$ such that $(\rho,\overline{J},\overline{E},\overline{B})$ satisfy Maxwell's equations and $\overline{B}=\overline{0}$. We then have that $\square^{2}(\overline{B})=\overline{0}$ and, by the second result in Lemma \ref{electricmagnetic}, that $\bigtriangledown\times \overline{J}=\overline{0}$. By the same argument, and using the definition of non-radiating, we must have that $\bigtriangledown'\times \overline{J}'=\overline{0}$ for the transformed current and charge $(\rho',\overline{J}')$ in any inertial frame $S'$ with velocity vector $\overline{v}$. We now compute $\bigtriangledown'\times\overline{J}'$. We have, as above, the transformation rule for $\overline{J}'$ given by;\\

$\overline{J}'=\gamma(\overline{J}_{||}-\rho\overline{v})+\overline{J}_{\perp}$\\

so that, using elementary vector calculus;\\

$\bigtriangledown'\times\overline{J}'=\gamma(\bigtriangledown'\times \overline{J}_{||})+\gamma(\overline{v}\times \bigtriangledown'(\rho))+\bigtriangledown'\times \overline{J}_{\perp}$\\

We also have, see \cite{lcl}, the transformation rule for $\bigtriangledown'$;\\

$\bigtriangledown'=\gamma(\bigtriangledown_{||}+{\overline{v}\over c^{2}}{\partial\over\partial t})+\bigtriangledown_{\perp}$\\

Taking $\overline{v}=v\hat{x}$, we have that $\bigtriangledown_{||}=({\partial\over \partial x},0,0)$, $\bigtriangledown_{\perp}=(0,{\partial\over \partial y},{\partial\over \partial z})$, ${\overline{v}\over c^{2}}{\partial\over\partial t}=({v\over c^{2}}{\partial\over\partial t},0,0)$, $\overline{J}_{||}=(J_{1},0,0)$ and $\overline{J}_{\perp}=(0,J_{2},J_{3})$ so that;\\

$\bigtriangledown'=({\partial\over \partial x}',{\partial\over \partial y'},{\partial\over \partial z'})$\\

$=\gamma({\partial\over \partial x},0,0)+\gamma({v\over c^{2}}{\partial\over\partial t},0,0)+(0,{\partial\over \partial y},{\partial\over \partial z})$\\

$=(\gamma{\partial\over \partial x}+{\gamma v\over c^{2}}{\partial\over\partial t},{\partial\over \partial y},{\partial\over \partial z})$\\

while;\\

$\gamma(\bigtriangledown'\times \overline{J}_{||})=\gamma(0,{\partial J_{1}\over \partial z'},-{\partial J_{1}\over \partial y'})$\\

$=(0,\gamma{\partial J_{1}\over \partial z},-\gamma{\partial J_{1}\over \partial y})$\\

and;\\

$\bigtriangledown'\times \overline{J}_{\perp}=({\partial J_{3}\over\partial y'}-{\partial J_{2}\over\partial z'},-{\partial J_{3}\over\partial x'},{\partial J_{2}\over\partial x'})$\\

$=({\partial J_{3}\over\partial y}-{\partial J_{2}\over\partial z},-\gamma{\partial J_{3}\over\partial x}-{\gamma v\over c^{2}}{\partial J_{3}\over\partial t},\gamma{\partial J_{2}\over\partial x}+{\gamma v\over c^{2}}{\partial J_{2}\over\partial t})$\\

and;\\

$\bigtriangledown'(\rho)=({\partial\rho\over \partial x'},{\partial\rho\over \partial y'},{\partial\rho\over \partial z'})$\\

$=(\gamma{\partial\rho\over \partial x}+{\gamma v\over c^{2}}{\partial\rho\over \partial t},{\partial\rho\over \partial y},{\partial\rho\over \partial z})$\\

$\gamma(\overline{v}\times\bigtriangledown'(\rho))=(0,-\gamma v{\partial\rho\over \partial z},\gamma v{\partial\rho\over \partial y})$\\

Combining these results, it follows that;\\

$\bigtriangledown'\times\overline{J}'=(0,\gamma{\partial J_{1}\over \partial z},-\gamma{\partial J_{1}\over \partial y})+(0,-\gamma v{\partial\rho\over \partial z},\gamma v{\partial\rho\over \partial y})$\\

$+({\partial J_{3}\over\partial y}-{\partial J_{2}\over\partial z},-\gamma{\partial J_{3}\over\partial x}-{\gamma v\over c^{2}}{\partial J_{3}\over\partial t},\gamma{\partial J_{2}\over\partial x}+{\gamma v\over c^{2}}{\partial J_{2}\over\partial t})$\\

$=({\partial J_{3}\over\partial y}-{\partial J_{2}\over\partial z},\gamma{\partial J_{1}\over \partial z}-\gamma{\partial J_{3}\over\partial x}-{\gamma v\over c^{2}}{\partial J_{3}\over\partial t}-\gamma v{\partial\rho\over \partial z},-\gamma{\partial J_{1}\over \partial y}+\gamma{\partial J_{2}\over\partial x}+{\gamma v\over c^{2}}{\partial J_{2}\over\partial t}+\gamma v{\partial\rho\over \partial y})$\\

As we have seen, $\bigtriangledown\times\overline{J}=\overline{0}$ in $S$. In coordinates, this implies that;\\

${\partial J_{3}\over\partial y}-{\partial J_{2}\over\partial z}={\partial J_{1}\over\partial z}-{\partial J_{3}\over\partial x}={\partial J_{2}\over\partial x}-{\partial J_{1}\over\partial y}=0$\\

It follows, using Lemma \ref{electricmagnetic}, that;\\

$\bigtriangledown'\times\overline{J}'=(0,-{\gamma v\over c^{2}}{\partial J_{3}\over\partial t}-\gamma v{\partial\rho\over \partial z},{\gamma v\over c^{2}}{\partial J_{2}\over\partial t}+\gamma v{\partial\rho\over \partial y})$\\

$=\gamma(\overline{v}\times (\bigtriangledown(\rho)+{1\over c^{2}}{\partial\overline{J}\over \partial t}))$\\

As $\overline{v}$ is arbitrary and $\bigtriangledown'\times\overline{J}'=\overline{0}$, we conclude that;\\

$(\bigtriangledown(\rho)+{1\over c^{2}}{\partial\overline{J}\over \partial t})=\overline{0}$, $(\dag)$\\

Taking the divergence $div$ of $(\dag)$ and using the continuity equation, we obtain that;\\

$div(\bigtriangledown(\rho))+{1\over c^{2}}{\partial div(\overline{J})\over \partial t}$\\

$=\bigtriangledown^{2}(\rho)-{1\over c^{2}}{\partial^{2}\rho\over \partial t^{2}}=0$\\

so that $\square^{2}(\rho)=0$. We can conclude as in Lemma \ref{waveequation}, using $(\dag)$, the continuity equation and the fact that $\bigtriangledown\times\overline{J}=\overline{0}$, that $\square^{2}(\overline{J})=\overline{0}$ as required.

\end{proof}

We make a further definition;\\

\begin{defn}
\label{nonradiating2}
Let $(\rho,\overline{J})$ be a charge distribution and current, satisfying the continuity equation in the rest frame $S$. Then, we say that $(\rho,\overline{J})$ is strongly non-radiating if in any inertial frame $S'$, with velocity vector $\overline{v}$, for the transformed current and charge $(\rho',\overline{J}')$, there exist electric and magnetic fields $(\overline{E}',\overline{B}')$ in $S'$ such that $(\rho',\overline{J}',\overline{E}',\overline{B}')$ satisfy Maxwell's equations in $S'$ and with $\overline{E}'=\overline{0}$. We say that $(\rho,\overline{J})$ is mixed non-radiating if in any inertial frame $S'$, with velocity vector $\overline{v}$, for the transformed current and charge $(\rho',\overline{J}')$, there exist electric and magnetic fields $(\overline{E}',\overline{B}')$ in $S'$ such that $(\rho',\overline{J}',\overline{E}',\overline{B}')$ satisfy Maxwell's equations in $S'$ and with either $\overline{E}'=\overline{0}$ or $\overline{B}'=\overline{0}$. We say that $(\rho,\overline{J})$ is Poynting non-radiating if in any inertial frame $S'$, with velocity vector $\overline{v}$, for the transformed current and charge $(\rho',\overline{J}')$, there exist electric and magnetic fields $(\overline{E}',\overline{B}')$ in $S'$ such that $(\rho',\overline{J}',\overline{E}',\overline{B}')$ satisfy Maxwell's equations in $S'$ and with the Poynting vector $\overline{E}'\times\overline{B}'=\overline{0}$. We say that $(\rho,\overline{J})$ is surface non-radiating if in any inertial frame $S'$, with velocity vector $\overline{v}$, for the transformed current and charge $(\rho',\overline{J}')$, there exist electric and magnetic fields $(\overline{E}',\overline{B}')$ in $S'$ such that $(\rho',\overline{J}',\overline{E}',\overline{B}')$ satisfy Maxwell's equations in $S'$ and with  $div(\overline{E}'\times\overline{B}')=0$.
\end{defn}

We note the following;\\

\begin{lemma}
\label{nononradiating2}
Let $(\rho,\overline{J})$ be strongly non-radiating, then $(\rho,\overline{J})$ is trivial, that is $\rho=0$ and $\overline{J}=\overline{0}$.
\end{lemma}

\begin{proof}
In the rest frame $S$, we can find a pair $(\overline{E},\overline{B})$ such that $(\rho,\overline{J},\overline{E},\overline{B})$ satisfy Maxwell's equations and with $\overline{E}=\overline{0}$. By Maxwell's equations, we have that $div(\overline{E})={\rho\over\epsilon_{0}}=0$, so that $\rho=0$, $(*)$. In an inertial frame $S'$, with velocity vector $\overline{v}$, we have, by the transformation rules and $(*)$, that;\\

$\rho'=\gamma(\rho-{vJ_{||}\over c^{2}})=-{\gamma vJ_{||}\over c^{2}}$, $(**)$\\

Using the fact that we can find $(\overline{E}',\overline{B}')$ such that $(\rho',\overline{J}',\overline{E}',\overline{B}')$ satisfy Maxwell's equations and with $\overline{E}'=\overline{0}$, we can conclude again, that $\rho'=0$. By $(**)$, we then have that $J_{||}=0$, so that $\overline{J},\overline{v}=0$, for every velocity vector $\overline{v}$. This clearly implies that $\overline{J}=\overline{0}$ as required.

\end{proof}

\begin{lemma}
\label{nononradiating3}
Let $(\rho,\overline{J})$ be mixed non-radiating, but not non-radiating, then $(\rho,\overline{J})$ is trivial, that is $\rho=0$ and $\overline{J}=\overline{0}$.
\end{lemma}

\begin{proof}
Without loss of generality, using the result of Lemma \ref{nonradiating2}, we can assume that in the rest frame $S$, there exists a pair $(\overline{E},\overline{B})$ such that $(\rho,\overline{J},\overline{E},\overline{B})$ satisfy Maxwell's equations and with $\overline{E}=\overline{0}$, and that there exists an inertial frame $S'$, with velocity vector $\overline{v}$, such that, for the transformed charge and current $(\rho',\overline{J}')$, there exists a pair $(\overline{E}',\overline{B}')$ such that $(\rho',\overline{J}',\overline{E}',\overline{B}')$ satisfy Maxwell's equations and with $\overline{B'}=\overline{0}$. Working in the rest frame $S$, we have that, by Maxwell's equations, $div(\overline{E})={\rho\over\epsilon_{0}}$, so that $\rho=0$. As $\square^{2}(\overline{E})=\overline{0}$, we have by Lemma \ref{electricmagnetic}, that;\\

$\mu_{0}{\partial\overline{J}\over \partial t}={\bigtriangledown(\rho)\over\epsilon_{0}}+\mu_{0}{\partial\overline{J}\over \partial t}=\overline{0}$\\

so that $\overline{J}$ is static. Again by Maxwell's equations, we have that;\\

$\bigtriangledown\times\overline{E}=-{\partial\overline{B}\over \partial t}$\\

$\bigtriangledown\times\overline{B}=\mu_{0}\overline{J}+\mu_{0}\epsilon_{0}{\partial\overline{E}\over \partial t}$\\

so that, as $\overline{E}=\overline{0}$, $\overline{B}$ is static, and $\bigtriangledown\times\overline{B}=\mu_{0}\overline{J}$, $(*)$. Switching to the frame $S'$, using the fact that $\overline{B}'=\overline{0}$ and the second result of Lemma \ref{electricmagnetic}, we have that $\bigtriangledown'\times\overline{J}'=\overline{0}$. Repeating the calculation of Lemma \ref{wavenonradiating}, and using the fact that $\square^{2}(\overline{E})=\overline{0}$, we have that;\\

$\bigtriangledown'\times\overline{J}'=(\bigtriangledown\times\overline{J})_{||}+\gamma((\bigtriangledown\times\overline{J})_{\perp})+\gamma(\overline{v}\times\square^{2}(\overline{E}))$\\

$=(\bigtriangledown\times\overline{J})_{||}+\gamma((\bigtriangledown\times\overline{J})_{\perp})=\overline{0}$\\

It follows that;\\

$((\bigtriangledown\times\overline{J}),\overline{v})$\\

$=(\bigtriangledown\times\overline{J})_{||},\overline{v})$\\

$=-\gamma((\bigtriangledown\times\overline{J})_{\perp},\overline{v})=\overline{0}$\\

so that;\\

$(\bigtriangledown\times\overline{J})_{||}=-\gamma(\bigtriangledown\times\overline{J})_{\perp}=\overline{0}$\\

and;\\

$(\bigtriangledown\times\overline{J})=(\bigtriangledown\times\overline{J})_{||}+(\bigtriangledown\times\overline{J})_{\perp}=\overline{0}$\\

By the second result of Lemma \ref{electricmagnetic}, we obtain that $\square^{2}(\overline{B})=\overline{0}$, but $\overline{B}$ is static, so in fact $\bigtriangledown^{2}(\overline{B})=\overline{0}$. Applying Liouville's theorem, and using the fact that $\overline{B}$ is bounded and vanishing at infinity, we obtain that $\overline{B}=\overline{0}$. From $(*)$, we must have that $\overline{J}=\overline{0}$ as well, proving the claim.

\end{proof}

\begin{rmk}
\label{conjectures}
We conjecture that if $(\rho,\overline{J})$ are Poynting or surface non-radiating, but not non-radiating, then $(\rho,\overline{J})$ are trivial. Given these conjectures, if an electromagnetic system fails to satisfy the wave equation outline above, then in some inertial frame, without loss of generality, we would have that $div(\overline{E}\times\overline{B})>0$ on some open $U$. By the divergence theorem, would imply an energy flux through the boundary of $U$. This imposes strong restrictions on the nature of this flux, as if the total energy $V$ of the system were to reduce to zero, or even decrease then, we can consider Rutherford's observation, that, in an atomic system, the orbiting electrons would spiral into the nucleus.\\

\end{rmk}
\end{section}
\begin{section}{The Balmer Series}
We now consider flows satisfying the wave equation.

\begin{lemma}
\label{waveequationlinks}
Let $(\rho,\overline{J})$ be a pair, satisfying the continuity equation, with $\rho\in S(\mathcal{R}^{3},\mathcal{R})$, $\overline{J}\in S((\mathcal{R}^{3},\mathcal{R}^{3}))$ and the wave equations $\square^{2}(\rho)=0$ and $\square^{2}(\overline{J})=\overline{0}$, with the additional equation;\\

$\bigtriangledown(\rho)+{1\over c^{2}}{\partial\overline{J}\over \partial t}=\overline{0}$ $(*)$\\

Then if;\\

$\rho(\overline{x},t)=\int_{\mathcal{R}^{3}}f(\overline{k})e^{i(\overline{k}\centerdot\overline{x}-\omega(\overline{k})t)}d\overline{k}+\int_{\mathcal{R}^{3}}g(\overline{k})e^{i(\overline{k}\centerdot\overline{x}+\omega(\overline{k})t)}d\overline{k}$\\

$\overline{J}(\overline{x},t)=\int_{\mathcal{R}^{3}}\overline{F}(\overline{k})e^{i(\overline{k}\centerdot\overline{x}-\omega(\overline{k})t)}d\overline{k}+\int_{\mathcal{R}^{3}}\overline{G}(\overline{k})e^{i(\overline{k}\centerdot\overline{x}+\omega(\overline{k})t)}d\overline{k}$\\

we have that, for $\overline{k}\neq\overline{0}$ ;\\

$\overline{F}(\overline{k})={cf(\overline{k})\overline{k}\over |\overline{k}|}$, $f(\overline{k})={(\overline{k},F(\overline{k}))\over c|\overline{k}|}$\\

$\overline{G}(\overline{k})=-{cg(\overline{k})\overline{k}\over |\overline{k}|}$, $g(\overline{k})=-{(\overline{k},G(\overline{k}))\over c|\overline{k}|}$\\

If $\overline{J}$ is tangential, that is for $\overline{x}\neq\overline{0}$, and $t\in\mathcal{R}_{\geq 0}$,  $(\overline{x},\overline{J}(\overline{x},t))=0$, then the pair $(\rho,\overline{J})$ is trivial, that is $\rho=0$ and $\overline{J}=\overline{0}$.\\

\end{lemma}

\begin{proof}
By the first part of Lemma \ref{solution}, using the fact that $\rho$ and $\overline{J}$ satisfy the wave equation, we can write;\\

$\rho(\overline{x},t)=\int_{\mathcal{R}^{3}}f(\overline{k})e^{i(\overline{k}\centerdot\overline{x}-\omega(\overline{k})t)}d\overline{k}+\int_{\mathcal{R}^{3}}g(\overline{k})e^{i(\overline{k}\centerdot\overline{x}+\omega(\overline{k})t)}d\overline{k}$\\

$\overline{J}(\overline{x},t)=\int_{\mathcal{R}^{3}}\overline{F}(\overline{k})e^{i(\overline{k}\centerdot\overline{x}-\omega(\overline{k})t)}d\overline{k}+\int_{\mathcal{R}^{3}}\overline{G}(\overline{k})e^{i(\overline{k}\centerdot\overline{x}+\omega(\overline{k})t)}d\overline{k}$\\

where $f,g\subset S(\mathcal{R}^{3},\mathcal{R})$, $F,G\subset S(\mathcal{R}^{3},\mathcal{R})$ and $\omega(\overline{k})=c|\overline{k}|$\\

We have that;\\

$\bigtriangledown(\rho)(\overline{x},t)=\int_{\mathcal{R}^{3}}f(\overline{k})i\overline{k}e^{i(\overline{k}\centerdot\overline{x}-\omega(\overline{k})t)}d\overline{k}+\int_{\mathcal{R}^{3}}g(\overline{k})i\overline{k}e^{i(\overline{k}\centerdot\overline{x}+\omega(\overline{k})t)}d\overline{k}$\\

while;\\

${\partial\overline{J}\over \partial t}(\overline{x},t)=\int_{\mathcal{R}^{3}}-i\omega(\overline{k})\overline{F}(\overline{k})e^{i(\overline{k}\centerdot\overline{x}-\omega(\overline{k})t)}d\overline{k}+\int_{\mathcal{R}^{3}}i\omega(\overline{k})\overline{G}(\overline{k})e^{i(\overline{k}\centerdot\overline{x}+\omega(\overline{k})t)}d\overline{k}$\\

so that, equating coefficients, using the Inversion Theorem, and $(*)$, we have that;\\

$f(\overline{k})i\overline{k}-{i\over c^{2}}\omega(\overline{k})\overline{F}(\overline{k})=\overline{0}$\\

$g(\overline{k})i\overline{k}+{i\over c^{2}}\omega(\overline{k})\overline{G}(\overline{k})=\overline{0}$\\

$\overline{F}(\overline{k})={cf(\overline{k})\overline{k}\over |\overline{k}|}$, $(\overline{k}\neq\overline{0})$\\

$\overline{G}(\overline{k})=-{cg(\overline{k})\overline{k}\over |\overline{k}|}$, $(\overline{k}\neq\overline{0})$ $(**)$\\

We have that;\\

${\partial\rho\over \partial t}(\overline{x},t)=\int_{\mathcal{R}^{3}}-i\omega(\overline{k})f(\overline{k})e^{i(\overline{k}\centerdot\overline{x}-\omega(\overline{k})t)}d\overline{k}+\int_{\mathcal{R}^{3}}i\omega(\overline{k})g(\overline{k})e^{i(\overline{k}\centerdot\overline{x}+\omega(\overline{k})t)}d\overline{k}$\\

and;\\

$div(\overline{J})(\overline{x},t)=\int_{\mathcal{R}^{3}}i(\overline{k},\overline{F}(\overline{k}))e^{i(\overline{k}\centerdot\overline{x}-\omega(\overline{k})t)}d\overline{k}+\int_{\mathcal{R}^{3}}i(\overline{k},\overline{G}(\overline{k}))e^{i(\overline{k}\centerdot\overline{x}+\omega(\overline{k})t)}d\overline{k}$\\

so that, equating coefficients again, and using the continuity equation ${\partial\rho\over \partial t}+div(\overline{J})=0$, we have;\\

$-i\omega(\overline{k})f(\overline{k})+i(\overline{k},F(\overline{k}))=0$\\

$i\omega(\overline{k})g(\overline{k})+i(\overline{k},G(\overline{k}))=0$\\

$f(\overline{k})={(\overline{k},F(\overline{k}))\over c|\overline{k}|}$, $(\overline{k}\neq\overline{0})$\\

$g(\overline{k})=-{(\overline{k},G(\overline{k}))\over c|\overline{k}|}$, $(\overline{k}\neq\overline{0})$ $(***)$\\

Now suppose that $\overline{J}$ is tangential. We then have, applying the fourier transform $\mathcal{F}$, see \cite{E};\\

$\mathcal{F}(x_{1}J_{1}+x_{2}J_{2}+x_{3}J_{3})$\\

$=-i({\partial\mathcal{F}(J_{1})\over \partial k_{1}}+{\partial\mathcal{F}(J_{2})\over \partial k_{2}}+{\partial\mathcal{F}(J_{3})\over \partial k_{3}})$\\

$=-i(div(\mathcal{F}(\overline{J})(\overline{k},t)))=0$\\

so that $div(\mathcal{F}(\overline{J})(\overline{k},t))=0$ which implies, equating coefficients, that $div(\overline{F})(\overline{k})=div(\overline{G})(\overline{k})=0$. Using the formula $(**)$, we have;\\

$div(\overline{F})(\overline{k})=div({cf(\overline{k})\overline{k}\over |\overline{k}|})$\\

$=c(\bigtriangledown(f)(\overline{k}),{\overline{k}\over |\overline{k}|})+cf(\overline{k})div({\overline{k}\over |\overline{k}|})$\\

$=c(\bigtriangledown(f)(\overline{k}),{\overline{k}\over |\overline{k}|})+cf(\overline{k}){2\over |\overline{k}|}=0$\\

so that;\\

$=(\bigtriangledown(f),\overline{k})=-2f$, for $\overline{k}\neq \overline{0}$\\

In coordinates, this would imply that;\\

${\partial f\over \partial k_{1}}={\partial f\over \partial k_{2}}={\partial f\over \partial k_{3}}={\partial f\over \partial k_{1}}+{\partial f\over \partial k_{2}}+{\partial f\over \partial k_{3}}=-2f$\\

so that $-6f=-2f$ and $f=0$. Similarly, we conclude that $g=0$, and, using the equations $(**)$, that $\overline{F}=\overline{G}=\overline{0}$. This implies that $\rho=0$ and $\overline{J}=\overline{0}$ as required.

\end{proof}
\begin{lemma}
\label{eigenvalues}
We can find a pair $(\rho,\overline{J})$ satisfying the hypotheses of Lemma \ref{waveequationlinks}, with the additional requirement that $\overline{J}|_{S(r_{0})}=\overline{0}$.

\end{lemma}
\begin{proof}
We convert to spherical polar coordinates, $x=rcos(\phi)sin(\theta)$, $y=rsin(\phi)sin(\theta)$, $z=rcos(\theta)$, for $0\leq\theta\leq\pi$, $-\pi\leq\phi\leq\pi$, writing the Laplacian;\\

$\bigtriangledown^{2}(u)=(R_{r}+{1\over r^{2}}A_{\theta,\phi})(u)$\\

where;\\

$R_{r}(u)={\partial^{2}u\over\partial r^{2}}+{2\over r}{\partial u\over \partial r}$\\

$A_{\theta,\phi}(u)={1\over\sin(\theta)}{\partial\over\partial\theta}(sin(\theta){\partial u\over\partial\theta})+{1\over sin^{2}\theta}{\partial^{2}u\over\partial\phi^{2}}$\\

are the radial and angular components respectively. The eigenvectors of the operator $A_{\theta,\phi}$ are the spherical harmonics, defined by;\\

$Y_{l,m}(\theta,\phi)=(-1)^{m}({2l+1\over 4\pi}{(l-m)!\over (l+m)!})^{1\over 2}P_{l,m}(cos(\theta))e^{im\phi}$, $(l\geq 0,0\leq m\leq l)$\\

$\overline{Y_{l,m}}(\theta,\phi)=(-1)^{m}Y_{l,-m}$, $(l\geq 0,0\leq m\leq l)$\\

where;\\

$P_{l,m}(x)=(1-x^{2})^{m\over 2}{d^{m}\over dx^{m}}(P_{l}(x))$, $(l\geq 0,0\leq m\leq l)$\\

$P_{l}(x)={1\over 2^{l}l!}{d^{l}\over dx^{l}}((x^{2}-1)^{l})$, $(l\geq 0)$\\

see the appendix of \cite{M}. We have that $\{Y_{l,m}:l\geq 0,-l\leq m\leq l\}$ forms a complete orthonormal basis of $L^{2}(S(1))$, and, moreover;\\

$A_{\theta,\phi}(Y_{l,m})=-l(l+1)Y_{l,m}$, $(l\geq 0,-l\leq m\leq l)$\\

see the appendix of \cite{M} again. We look for eigenvectors of $\bigtriangledown^{2}$ of the form $\psi_{l,m,E}(r,\theta,\phi)=Y_{l,m}(\theta,\phi)\chi_{l,E}(r)$. We have, using $(\dag)$, that;\\

$\bigtriangledown^{2}(\psi_{l,m,E})=(R_{r}+{1\over r^{2}}A_{\theta,\phi})(Y_{l,m}(\theta,\phi)\chi_{l,E}(r))$\\

$=Y_{l,m}(\theta,\phi)R_{r}(\chi_{l,E}(r))-{l(l+1)\over r^{2}}Y_{l,m}(\theta,\phi)\chi_{l,E}(r))$\\

so that;\\

$\bigtriangledown^{2}(\psi_{l,m,E})=E\psi_{l,m,E}$\\

iff $\chi_{l,E}(r)$ satisfies the radial equation;\\

$(R_{r}-{l(l+1)\over r^{2}}-E)\chi_{l,E}(r)=0$\\

$({d^{2}\over d r^{2}}+{2\over r}{d\over dr}-{l(l+1)\over r^{2}}-E)\chi_{l,E}(r)=0$, $(\dag\dag)$\\

We can solve $(\dag\dag)$ using the method of Frobenius, see \cite{BP}, but the solutions are only bounded for $E<0$. Explicitly, taking $E=-k^{2}$, with $k>0$, and making the change of variables $s=kr$, the radial equation reduces to the spherical Bessel equation;\\

$({d^{2}\over ds^{2}}+{2\over s}{d\over ds}+1-{l(l+1)\over s^{2}})\chi_{l,E}(s)=0$, $(\dag\dag\dag)$\\

which, as noted in \cite{M} has a unique bounded solution (up to scalar multiplication) on $(0,\infty)$ defined by;\\

$j_{l}(s)=({\pi\over 2s})^{1\over 2}J_{l+{1\over 2}}(s)$\\

where $J_{l+{1\over 2}}(s)$ denotes the ordinary (of the first kind) Bessel function of order $l+{1\over 2}$. As is shown in \cite{M} again, see also \cite{L}, the functions;\\

$k({2\over \pi})^{1\over 2}Y_{l,m}(\theta,\phi)j_{l}(kr)$ $(k\in (0,\infty))$\\

form a complete orthonormal set in $C(R^{3})$. Moreover, we have the explicit representations;\\

$J_{1\over 2}(s)=({2\over \pi s})^{1\over 2}sin(s)$\\

$J_{l+{1\over 2}}(s)=({2\over \pi s})^{1\over 2}(P_{l}({1\over s})sin(s)-Q_{l-1}({1\over s})cos(s))$, $(l\in\mathcal{Z}_{\geq 1})$\\

where $\{P_{l},Q_{l}\}\subset \mathcal{Q}[x]$ are polynomials of degree $l$, with the property that $P_{l}(-1)=(-1)^{l}P_{l}(1)$ and $Q_{l}(-1)=(-1)^{l}Q_{l}(1)$, for $l\geq 0$, see \cite{JD}.\\

We set;\\

$\chi_{l,E}(r)=\tau_{l,k}(r)=k({2\over \pi})^{1\over 2}j_{l}(kr)$\\

$\gamma_{l,m,k}(r,\theta,\phi)=Y_{l,m}(\theta,\phi)\tau_{l,k}(r)$\\

where $E=-k^{2}$ for $k>0$. By what has been shown $\{\gamma_{l,m,k}:k\in (0,\infty),l\geq 0,-l\leq m\leq l\}$ forms a complete orthonormal set, $(*)$, and $\bigtriangledown^{2}(\gamma_{l,m,k})=-k^{2}\gamma_{l,m,k}$. It follows easily, that we can write a general solution for the charge $\rho$ and current contributions $\overline{J}$ in the wave equation using the forms;\\

$\rho=\sum_{l\geq 0}\sum_{-l\leq m\leq l}\int_{k>0}(u(l,m,k)\gamma_{l,m,k}e^{-ickt}+v(l,m,k)\gamma_{l,m,k}e^{ickt})dk$\\

$\overline{J}=\sum_{l\geq 0}\sum_{-l\leq m\leq l}\int_{k>0}(\overline{U}(l,m,k)\gamma_{l,m,k}e^{-ickt}+\overline{V}(l,m,k)\gamma_{l,m,k}e^{ickt})dk$, $(\dag\dag\dag\dag)$\\

We say that $\overline{J}$ satisfies the radial transform condition, if, in the notation of Lemma \ref{waveequationlinks}, we have that, for $\overline{k}\neq\overline{0}$;\\

$\overline{F}(\overline{k})=\alpha(k)\overline{k}$\\

$\overline{G}(\overline{k})=\beta(k)\overline{k}$\\

for some $\{\alpha,\beta\}\subset S(\mathcal{R}_{>0})$. As is easily shown, if $\overline{J}$ satisfies the radial transform condition, then if we define $\rho$ according to the second pair of equations in Lemma \ref{waveequationlinks}, we automatically have that $\overline{J}$ satisfies the first pair, and all the assumptions of Lemma \ref{waveequationlinks} are met. By considering the representation of $\overline{J}$ in Lemma \ref{waveequationlinks}, equating coefficients, and applying the inversion theorem, we see that;\\

$2ickF(\overline{k})=ick\int_{\overline{R}^{3}}\overline{J}_{0}e^{-i\overline{k}\centerdot\overline{x}}d\overline{x}-\int_{\overline{R}^{3}}({\partial\overline{J}\over \partial t})_{0}e^{-i\overline{k}\centerdot\overline{x}}d\overline{x}$\\

$2ickG(\overline{k})=ick\int_{\overline{R}^{3}}\overline{J}_{0}e^{-i\overline{k}\centerdot\overline{x}}d\overline{x}+\int_{\overline{R}^{3}}({\partial\overline{J}\over \partial t})_{0}e^{-i\overline{k}\centerdot\overline{x}}d\overline{x}$ $(**)$\\

for $\overline{k}\neq\overline{0}$. We compute these integrals using the representation of $\overline{J}$ in $(\dag\dag\dag\dag)$ and the representation;\\

$e^{i\overline{k}\centerdot\overline{x}}=4\pi\sum_{l=0}^{\infty}\sum_{m=-l}^{l}i^{l}j_{l}(kx)Y_{l,m}(\hat{\overline{k}})Y_{l,m}(\hat{\overline{x}})$\\

$=4\pi\sum_{l=0}^{\infty}\sum_{m=-l}^{l}{i^{l}\over k}({\pi\over 2})^{1\over 2}Y_{l,m}(\hat{\overline{k}})\gamma_{l,m,k}$\\

given in \cite{M}, where $k=|\overline{k}|$. We have, using the property $(*)$, that;\\

$\int_{\overline{R}^{3}}\overline{J}_{0}e^{-i\overline{k}\centerdot\overline{x}}d\overline{x}$\\

$=\int_{\overline{R}^{3}}(\sum_{l\geq 0}\sum_{-l\leq m\leq l}\int_{k>0}(\overline{U}(l,m,k)\gamma_{l,m,k}+\overline{V}(l,m,k)\gamma_{l,m,k})dk)e^{-i\overline{k}\centerdot\overline{x}}d\overline{x}$\\

$=4\pi\int_{\overline{R}^{3}}(\sum_{l\geq 0}\sum_{-l\leq m\leq l}\int_{k>0}(\overline{U}(l,m,k)\gamma_{l,m,k}+\overline{V}(l,m,k)\gamma_{l,m,k})dk)$\\

$(\sum_{l=0}^{\infty}\sum_{m=-l}^{l}{i^{-l}\over k}({\pi\over 2})^{1\over 2}Y_{l,m}^{*}(\hat{\overline{k}})\gamma_{l,m,k}^{*})d\overline{x}$\\

$=4\pi\sum_{l=0}^{\infty}\sum_{m=-l}^{l}(\overline{U}(l,m,k)+\overline{V}(l,m,k)){i^{-l}\over k}({\pi\over 2})^{1\over 2}Y_{l,m}^{*}(\hat{\overline{k}})$\\

A similar calculation shows that;\\

$\int_{\overline{R}^{3}}({\partial\overline{J}\over \partial t})_{0}e^{-i\overline{k}\centerdot\overline{x}}d\overline{x}$\\

$=4\pi\sum_{l=0}^{\infty}\sum_{m=-l}^{l}(-ick\overline{U}(l,m,k)+ick\overline{V}(l,m,k)){i^{-l}\over k}({\pi\over 2})^{1\over 2}Y_{l,m}^{*}(\hat{\overline{k}})$\\

It follows from $(**)$ that;\\

$\overline{F}(\overline{k})=4\pi\sum_{l=0}^{\infty}\sum_{m=-l}^{l}\overline{U}(l,m,k){i^{-l}\over k}({\pi\over 2})^{1\over 2}Y_{l,m}^{*}(\hat{\overline{k}})$\\

$\overline{G}(\overline{k})=4\pi\sum_{l=0}^{\infty}\sum_{m=-l}^{l}\overline{V}(l,m,k){i^{-l}\over k}({\pi\over 2})^{1\over 2}Y_{l,m}^{*}(\hat{\overline{k}})$ $(++)$\\

We can compute $\overline{k}$ in spherical harmonics by;\\

$\overline{k}=\overline{k}^{*}=k(cos(\phi)sin(\theta),sin(\phi)sin(\theta),cos(\theta))$\\

$=\sum_{l=0}^{\infty}\sum_{m=-l}^{l}k\overline{W}(l,m)^{*}Y_{l,m}^{*}(\hat{\overline{k}})$\\

noting that, by orthonormality of the spherical harmonics;\\

$\sum_{l=0}^{\infty}\sum_{m=-l}^{l}|\overline{W}(l,m)|^{2}=4\pi$ $(+)$\\

Equating coefficients, the radial transform condition is satisfied setting;\\

$\overline{U}(l,m,k)=\alpha(k)({2\over \pi})^{1\over 2}{i^{l}k^{2}\over 4\pi}\overline{W}(l,m)^{*}$\\

$\overline{V}(l,m,k)=\beta(k)({2\over \pi})^{1\over 2}{i^{l}k^{2}\over 4\pi}\overline{W}(l,m)^{*}$ $(***)$\\

where $\{\alpha,\beta\}\subset S(\mathcal{R}_{>0})$, and $l\geq 0$, $-l\leq m\leq l$, $k>0$.\\

We now impose the boundary condition, that $\overline{J}|_{S(r_{0})}=\overline{0}$. We can achieve this by requiring that $\gamma_{l,m,k}|_{S(r_{0})}=0$, or equivalently, that $\tau_{k,l}(r_{0})=0$, or $j_{l}(kr_{0})=0$. The positive zeros $S_{l}$ of $j_{l}$ form a discrete set and we require that $k\in {S_{l}\over r_{0}}$. Using the asymptotic approximation;\\

$j_{l}(s)\sim_{s\rightarrow\infty}{sin(s-{l\pi\over 2})\over s}$\\

given in \cite{M}, we have;\\

$k\sim {\pi\over r_{0}}(n+{l\over 2})$, $n\in\mathcal{Z}_{>0}$\\

for large values of $k$. Using $(\dag\dag\dag\dag)$, we have that $\overline{J}$ takes the form;\\

$\overline{J}=\sum_{l\geq 0}\sum_{-l\leq m\leq l}\sum_{k\in {S_{l}\over r_{0}} }(\overline{U_{1}}(l,m,k)\gamma_{l,m,k}e^{-ickt}+\overline{V_{1}}(l,m,k)\gamma_{l,m,k}e^{ickt})$ $(\dag\dag\dag\dag\dag)$\\

where we have that;\\

$\overline{U}=\sqrt{\eta}\overline{U_{1}}$\\

$\overline{V}=\sqrt{\eta}\overline{V_{1}}$, $(\sharp)$\\

for some nonstandard infinite $\eta$ and the coefficients $\{\overline{U},\overline{V}\}$ are chosen to satisfy $(***)$ at the discrete eigenvalues.

\end{proof}

\begin{rmk}
\label{alternating}
Technically, the calculation $(++)$ requires smoothness of the coefficients $\{\overline{U},\overline{V}\}$ in the continuous variable $k$, so that we can invoke the Riemann-Lebsegue lemma, to eliminate the orthogonal terms $k\neq k'$. When passing to a discrete sum, we lose this property, and, an argument involving equating coefficients is required. We have sketched over this by involving a nonstandard element $\eta$, but, if the reader is unfamiliar with this circle of ideas, we are essentially using distributions. As this is primarily a Physics paper, we leave the technical details for another occasion.
\end{rmk}

\begin{lemma}
\label{sphereorthogonal}
Let $s_{l,k}(r)={\tau_{l,k}(r)\over c_{l,k}}$, where $c_{l,k}={k^{1\over 2}r_{0}\over \sqrt{2}}J_{l+{3\over 2}}(kr_{0})$, then, for $k\in {S_{l}\over r_{0}}$, $l\geq 0$, $l$ fixed, $s_{k,l}$ forms a complete orthonormal system in $C_{0,2}((0,r_{0}))$, consisting of continuous functions on the interval $(0,r_{0})$, vanishing at $r_{0}$, with respect to the measure $r^{2}dr$. Moreover, the functions $\delta_{l,m,k}(r,\theta,\phi)=Y_{l,m}(\theta,\phi)s_{l,k}(r)$, for $l\geq 0$,$-l\leq m\leq l$, $k\in {S_{l}\over r_{0}}$ form a complete orthonormal system in $C_{0,2}(B(r_{0}))$, consisting of continuous functions on the ball $B(r_{0})$ of radius $r_{0}$, vanishing at the boundary $S(r_{0})$, with respect to the standard measure $dB$.
\end{lemma}

\begin{proof}
Let;\\

$M_{l,r}={d^{2}\over dr^{2}}+{2\over r}{d\over dr}-{l(l+1)\over r^{2}}$\\

$L_{l,r}=r^{2}M_{l,r}=r^{2}{d^{2}\over dr^{2}}+2r{d\over dr}-l(l+1)$\\

so that;\\

$L_{l,r}(f)=-(-r^{2}f')'-l(l+1)f$\\

for $f\in C^{2}(0,r_{0})$. By Lagrange's identity, see \cite{BP}, we have that;\\

$\int_{0}^{r_{0}}(L_{l,r}(u)v-uL_{l,r}(v))dr=-(-r^{2}(u'v-uv'))|_{0}^{r_{0}}$ $(*)$\\

and, with notation as above, we have that;\\

$\tau_{l,k}={k^{1\over 2}\over r^{1\over 2}}J_{l+{1\over 2}}(kr)$\\

As $M_{l,r}(\tau_{l,k})=-k^{2}\tau_{l,k}$, we have applying $(*)$, that;\\

$(k'^{2}-k^{2})\int_{0}^{r_{0}}\tau_{l,k}\overline{\tau_{l,k'}}r^{2}dr$\\

$=(r^{2}(\tau_{l,k}'\tau_{l,k'}-\tau_{l,k}\tau_{l,k'}'))|_{0}^{r_{0}}$\\

$=r_{0}^{2}(({-k^{1\over 2}\over 2r_{0}^{3\over 2}}J_{l+{1\over 2}}(kr_{0})+{k^{3\over 2}\over r_{0}^{1\over 2}}J'_{l+{1\over 2}}(kr_{0})){k'^{1\over 2}\over r_{0}^{1\over 2}}J_{l+{1\over 2}}(k'r_{0})$\\

$-({-k'^{1\over 2}\over 2r_{0}^{3\over 2}}J_{l+{1\over 2}}(k'r_{0})+{k'^{3\over 2}\over r_{0}^{1\over 2}}J'_{l+{1\over 2}}(k'r_{0})){k^{1\over 2}\over r_{0}^{1\over 2}}J_{l+{1\over 2}}(kr_{0}))$\\

$={r_{0}^{2}\over r_{0}}(k^{3\over 2}k'^{1\over 2}J'_{l+{1\over 2}}(kr_{0})J_{l+{1\over 2}}(k'r_{0})-k'^{3\over 2}k^{1\over 2}J'_{l+{1\over 2}}(k'r_{0})J_{l+{1\over 2}}(kr_{0}))$\\

Clearly, if $\{k,k'\}\subset {S_{l}\over r_{0}}$ are distinct, this proves that $\tau_{l,k}$ and $\tau_{l,k'}$ are orthogonal with respect to the measure $r^{2}dr$. We then have, using l'Hospital's rule, assuming that $k\in {S_{l}\over r_{0}}$ and the recurrence relation for Bessel functions, see \cite{M};\\

$||\tau_{l,k}||^{2}_{r^{2}dr}=lim_{k'\rightarrow k}{r_{0}(k^{3\over 2}k'^{1\over 2}J'_{l+{1\over 2}}(kr_{0})J_{l+{1\over 2}}(k'r_{0})-k'^{3\over 2}k^{1\over 2}J'_{l+{1\over 2}}(k'r_{0})J_{l+{1\over 2}}(kr_{0}))\over (k'+k)(k'-k)}$\\

$={r_{0}\over 2}lim_{k'\rightarrow k}{(kJ'_{l+{1\over 2}}(kr_{0})J_{l+{1\over 2}}(k'r_{0})-k'J'_{l+{1\over 2}}(k'r_{0})J_{l+{1\over 2}}(kr_{0}))\over (k'-k)}$\\

$={r_{0}\over 2}(kr_{0}J'_{l+{1\over 2}}(kr_{0})J'_{l+{1\over 2}}(kr_{0})-J'_{l+{1\over 2}}(kr_{0})J_{l+{1\over 2}}(kr_{0})-kr_{0}J''_{l+{1\over 2}}(kr_{0})J_{l+{1\over 2}}(kr_{0}))$\\

$={kr_{0}^{2}\over 2}[J'_{l+{1\over 2}}(kr_{0})]^{2}$\\

$={kr_{0}^{2}\over 2}[J_{l+{3\over 2}}(kr_{0})]^{2}$\\

It follows immediately, that, for fixed $l\geq 0$, the $s_{l,k}$ form an orthonormal system. The proof that the $s_{l,k}$ form a complete system is sketched in \cite{W}. As $\{Y_{l,m}:l\geq 0,-l\leq m\leq l\}$ forms an orthogonal system on $S(1)$, we have that;\\

$\int_{S(r_{0})}\delta_{l,m,k}\overline{\delta_{l',m',k'}}dB$\\

$=\int_{S(r_{0})}Y_{l,m}s_{l,k}\overline{Y_{l',m'}s_{l',k'}}dB$\\

$=\int_{0}^{r_{0}}\int_{S(1)}Y_{l,m}\overline{Y_{l',m'}}(\theta,\phi)s_{l,k}\overline{s_{l',k'}}(r)r^{2}dS(1)dr$\\

$=\delta_{l,m}\delta_{l,k}$\\

proving that the $\delta_{l,m,k}$ form an orthonormal system. Completeness then follows easily from completeness of the $Y_{l,m}$ and the $s_{l,k}$.

\end{proof}

\begin{lemma}
\label{balmer}
For the fundamental electric field solutions $\overline{E}_{l_{0},k_{0}}^{\alpha,\beta}$, as defined below, the corresponding time averaged energies $<U_{em,l_{0},k_{0}}^{Q}>$, determined by the conserved quantity $Q\neq 0$, defined below, are quantised and display the properties of the Balmer series. Moreover, for a general bounded electric field solution $\overline{E}$, determined by $(\rho,\overline{J})$, satisfying the hypotheses of Lemma \ref{eigenvalues}, the corresponding energy $U_{em}$ can be computed in terms of the fundamental energies.
\end{lemma}

\begin{proof}

 We compute the electric field $\overline{E}$, assuming the magnetic field $\overline{B}$ vanishes. By Maxwell's equations;\\

${\partial \overline{E}\over \partial t}=-{1\over\epsilon_{0}}\overline{J}$\\

so that, integrating $(\dag\dag\dag\dag\dag)$ of Lemma \ref{eigenvalues}, requiring the boundedness condition, using the result of Lemma \ref{sphereorthogonal}, and the relations, $(\sharp)$ of Lemma \ref{eigenvalues}, we have;\\

$\overline{E}={-1\over\epsilon_{0}}\sum_{l\geq 0}\sum_{-l\leq m\leq l}\sum_{k\in {S_{l}\over r_{0}} }({\overline{U_{1}}(l,m,k)\over -ick}\gamma_{l,m,k}e^{-ickt}+{\overline{V_{1}}(l,m,k)\over ick}\gamma_{l,m,k}e^{ickt})$\\

$={-1\over\epsilon_{0}\sqrt{\eta}}\sum_{l\geq 0}\sum_{-l\leq m\leq l}\sum_{k\in {S_{l}\over r_{0}} }({\overline{U}(l,m,k)c_{l,k}\over -ick}\delta_{l,m,k}e^{-ickt}+{\overline{V}(l,m,k)c_{l,k}\over ick}\delta_{l,m,k}e^{ickt})$\\

$(\dag\dag\dag\dag\dag\dag)$\\

From here, we rely on the fact, proved in \cite{W}, that for $\{l_{1},l_{2}\}\subset\mathcal{Z}_{\geq 0}$ distinct, the Bessel functions $J_{l_{1}+{1\over 2}}$ and $J_{l_{2}+{1\over 2}}$ have no common zeros. We define the fundamental solutions $\overline{E}_{l_{0},k_{0}}^{\alpha,\beta}$, $l_{0}\geq 0$, $k_{0}\in S_{l_{0}}$ by requiring that $\alpha$ and $\beta$ are both supported at a single point $k_{0}\in S_{l_{0}}$ of the discrete union $\bigcup_{l\geq 0}S_{l}$, so that;\\

$\overline{E}_{l_{0},k_{0}}^{\alpha,\beta}={-1\over\epsilon_{0}\sqrt{\eta}}\sum_{-l_{0}\leq m\leq l_{0}}({\overline{U}(l_{0},m,k_{0})c_{l_{0},k_{0}}\over -ick_{0}}\delta_{l_{0},m,k_{0}}e^{-ick_{0}t}+{\overline{V}(l_{0},m,k_{0})c_{l_{0},k_{0}}\over ick_{0}}\delta_{l_{0},m,k_{0}}e^{ick_{0}t})$, $(\sharp\sharp)$\\

and both $\overline{U}(l_{0},m,k_{0})$ and $\overline{V}(l_{0},m,k_{0})$ are defined by $(***)$, in Lemma \ref{eigenvalues}. By Poynting's Theorem, see \cite{G}, using the facts $(*)$, $(***)$ of Lemma \ref{eigenvalues}, and the coefficient relations in Lemmas \ref{waveequationlinks} and \ref{sphereorthogonal}, the total energy stored in the electric field $\overline{E}_{l_{0},k_{0}}^{\alpha,\beta}$, restricted to $B(r_{0})$, is given by;\\

$U_{em,l_{0},k_{0}}^{\alpha,\beta}={\epsilon_{0}\over 2}\int_{B(r_{0})}|\overline{E}_{l_{0},k_{0}}^{\alpha,\beta}|^{2}d\overline{x}$\\

$={\epsilon_{0}\over 2}\int_{B(r_{0})}(\overline{E}_{l_{0},k_{0}}^{\alpha,\beta},\overline{E}_{l_{0},k_{0}}^{\alpha,\beta})d\overline{x}$\\

$\noindent ={\epsilon_{0}\over 2}{1\over\eta\epsilon_{0}^{2}}\sum_{-l_{0}\leq m\leq l_{0}}{c_{l_{0},k_{0}}^{2}\over c^{2}k_{0}^{2}}(|\overline{U}(l_{0},m,k_{0})|^{2}+|\overline{V}(l_{0},m,k_{0})|^{2}-2Re((\overline{U}(l_{0},m,k_{0}),\overline{V}(l_{0},m,k_{0}))e^{-2ick_{0}t}))$\\

$={1\over 2\eta\epsilon_{0}}\sum_{-l_{0}\leq m\leq l_{0}}{c_{l_{0},k_{0}}^{2}\over c^{2}k_{0}^{2}}(|\alpha(k_{0})|^{2}+|\beta(k_{0})|^{2}-2Re(\alpha(k_{0})\beta(k_{0})^{*}e^{-2ick_{0}t})){k_{0}^{4}\over 8\pi^{3}}|\overline{W}(l_{0},m)|^{2}$\\

$={r_{0}^{2}k_{0}^{3}J^{2}_{l_{0}+{3\over 2}}(k_{0}r_{0})\over 32\epsilon_{0}c^{2}\pi^{3}\eta}(|\alpha(k_{0})|^{2}+|\beta(k_{0})|^{2}-2Re(\alpha(k_{0})\beta(k_{0})^{*}e^{-2ick_{0}t}))\sum_{-l_{0}\leq m\leq l_{0}}|\overline{W}(l_{0},m)|^{2}$\\

$={r_{0}^{2}k_{0}^{3}J^{2}_{l_{0}+{3\over 2}}(k_{0}r_{0})\beta_{l_{0}}\over 32\epsilon_{0}c^{2}\pi^{3}\eta}(|{cf(k_{0})\over k_{0}}|^{2}+|{-cg(k_{0})\over k_{0}}|^{2}-2Re({cf(k_{0})\over k_{0}}{-cg(k_{0})\over k_{0}}^{*}e^{-2ick_{0}t}))$\\

$={r_{0}^{2}k_{0}J^{2}_{l_{0}+{3\over 2}}(k_{0}r_{0})\beta_{l_{0}}\over 32\epsilon_{0}\pi^{3}\eta}(|f(k_{0})|^{2}+|g(k_{0})|^{2}+2Re(f(k_{0})g(k_{0})^{*}e^{-2ick_{0}t}))$ $(\sharp\sharp\sharp\sharp\sharp)$\\

where $\beta_{l_{0}}=\sum_{-l_{0}\leq m\leq l_{0}}|\overline{W}(l_{0},m)|^{2}$\\

Now let;\\

$Q_{t}^{\alpha,\beta}=\int_{B(r_{0}}\rho_{t}^{\alpha,\beta}d\overline{x}$\\

Note that $Q_{t}^{\alpha,\beta}$ is conserved, as, by the continuity equation, the divergence theorem, and the vanishing of $\overline{J}_{t}^{\alpha,\beta}$ on $S(r_{0})$;\\

${dQ_{t}^{\alpha,\beta}\over dt}=\int_{B(r_{0})}{\partial \rho_{t}^{\alpha,\beta}\over \partial t}d\overline{x}$\\

$=\int_{B(r_{0})}-div(\overline{J}_{t}^{\alpha,\beta})d\overline{x}$\\

$=\int_{S(r_{0})}-\overline{J}_{t}^{\alpha,\beta}\centerdot d\overline{S}_{r_{0}}=0$\\

Using the relations $f(\overline{k})={\alpha(k)k\over c}$ and $g(\overline{k})=-{\beta(k)k\over c}$ from Lemma \ref{waveequationlinks} and the radial condition, we have, using the integral representation in Lemma \ref{waveequationlinks}, that for a fundamental solution;\\

$\rho_{l_{0},k_{0}}^{\alpha,\beta}={1\over \sqrt{\eta}}\int_{\overline{k}\in S(k_{0})}{\alpha(k_{0})k_{0}\over c}e^{i(\overline{k}\centerdot\overline{x}-k_{0}t)}-{\beta(k_{0})k_{0}\over c}e^{i(\overline{k}\centerdot\overline{x}+k_{0}t)}dS_{k_{0}}$\\

${\partial \rho_{l_{0},k_{0}}^{\alpha,\beta}\over \partial t}={1\over \sqrt{\eta}}\int_{\overline{k}\in S(k_{0})}{-i\alpha(k_{0})k_{0}^{2}\over c}e^{i(\overline{k}\centerdot\overline{x}-k_{0}t)}-{i\beta(k_{0})k_{0}^{2}\over c}e^{i(\overline{k}\centerdot\overline{x}+k_{0}t)}dS_{k_{0}}$\\

We can then use this representation, the result in \cite{m}, together with the conservation property, to obtain;\\

$Q=\int_{B(r_{0})}\rho_{l_{0},k_{0}}^{\alpha,\beta}d\overline{x}$\\

$={1\over \sqrt{\eta}}\int_{\overline{k}\in S(k_{0})}{\alpha(k_{0})k_{0}\over c}({2\pi r_{0}\over k_{0}})^{3\over 2}J_{3\over 2}(k_{0}r_{0})-{\beta(k_{0})k_{0}\over c}({2\pi r_{0}\over k_{0}})^{3\over 2}J_{3\over 2}(k_{0}r_{0})dS_{k_{0}}$\\

$={(\alpha(k_{0})-\beta(k_{0}))\over \sqrt{\eta}}{(2\pi r_{0})^{3\over 2}\over ck_{0}^{3\over 2}}k_{0}4\pi k_{0}^{2}J_{3\over 2}(k_{0}r_{0})$\\

$={(\alpha(k_{0})-\beta(k_{0}))\over \sqrt{\eta}}{(2\pi r_{0})^{3\over 2}\over c}4\pi k_{0}^{3\over 2}J_{3\over 2}(k_{0}r_{0})$\\

$0=\int_{B(r_{0})}{\partial \rho_{l_{0},k_{0}}^{\alpha,\beta}\over \partial t}d\overline{x}$\\

$={-i(\alpha(k_{0})+\beta(k_{0}))\over \sqrt{\eta}}{(2\pi r_{0})^{3\over 2}\over c}4\pi k_{0}^{5\over 2}J_{3\over 2}(k_{0}r_{0})$\\

so that, rearranging $\alpha(k_{0})=-\beta(k_{0})$ $f(k_{0})=g(k_{0})$, and, for $l_{0}\neq 1$;\\

$\alpha(k_{0})={Q\sqrt{\eta}c\over 8\pi(2\pi r_{0})^{3\over 2}k_{0}^{3\over 2}J_{3\over 2}(k_{0}r_{0})}$\\

$f(k_{0})={Q\sqrt{\eta}\over 8\pi(2\pi r_{0})^{3\over 2}k_{0}^{1\over 2}J_{3\over 2}(k_{0}r_{0})}$\\

Now we can substitute in $(\sharp\sharp\sharp\sharp\sharp)$, to obtain;\\

$U_{em,l_{0},k_{0}}^{Q}={r_{0}^{2}k_{0}J^{2}_{l_{0}+{3\over 2}}(k_{0}r_{0})\beta_{l_{0}}\over 32\epsilon_{0}\pi^{3}\eta}({2Q^{2}\eta(1+cos(2ck_{0}t))\over 64\pi^{2}(2\pi r_{0})^{3}k_{0}J^{2}_{3\over 2}(k_{0}r_{0})})$\\

$={Q^{2}\beta_{l_{0}}(1+cos(2ck_{0}t))\over 1024\pi^{8}\epsilon_{0}r_{0}}{J^{2}_{l_{0}+{3\over 2}}(k_{0}r_{0})\over J^{2}_{3\over 2}(k_{0}r_{0})}$\\

and, taking the average over a cycle;\\

$<U_{em,l_{0},k_{0}}^{Q}>={Q^{2}\beta_{l_{0}}\over 1024\pi^{8}\epsilon_{0}r_{0}}{J^{2}_{l_{0}+{3\over 2}}(k_{0}r_{0})\over J^{2}_{3\over 2}(k_{0}r_{0})}$\\

By the explicit representation of Bessel functions in Lemma \ref{eigenvalues}, we have that;\\

${J^{2}_{l_{0}+{3\over 2}}(k_{0}r_{0})\over J^{2}_{3\over 2}(k_{0}r_{0})}={(P_{l_{0}+1}({1\over k_{0}r_{0}})sin(k_{0}r_{0})-Q_{l_{0}}({1\over k_{0}r_{0}})cos(k_{0}r_{0}))^{2}\over (P_{1}({1\over k_{0}r_{0}})sin(k_{0}r_{0})-Q_{0}cos(k_{0}r_{0}))^{2}}$\\

and, using the asymptotic description of ${S_{l_{0}}\over r_{0}}$ for large values of $k_{0}$, in Lemma \ref{eigenvalues}, we have that;\\

$cos(k_{0}r_{0})\simeq (-1)^{n_{0}}(-1)^{l_{0}\over 2}$ $(l_{0}\ even)$\\

$cos(k_{0}r_{0})\simeq 0$ $(l_{0}\ odd)$\\

$sin(k_{0}r_{0})\simeq 0$ $(l_{0}\ even)$\\

$sin(k_{0}r_{0})\simeq (-1)^{n_{0}}(-1)^{l_{0}-1\over 2}$ $(l_{0}\ odd)$\\

so that, for $l_{0}$ even;\\

${J^{2}_{l_{0}+{3\over 2}}(k_{0}r_{0})\over J^{2}_{3\over 2}(k_{0}r_{0})}\simeq {Q_{l_{0}}^{2}({1\over k_{0}r_{0}})\over Q_{0}^{2}}={Q_{l_{0},0}^{2}\over Q_{0}^{2}}+{2Q_{l_{0},0}Q_{l_{0},2}\over Q_{0}^{2}k_{0}^{2}r_{0}^{2}}+O({1\over k_{0}^{4}r_{0}^{4}})$\\

and, for  $l_{0}$ odd;\\

${J^{2}_{l_{0}+{3\over 2}}(k_{0}r_{0})\over J^{2}_{3\over 2}(k_{0}r_{0})}\simeq {P_{l_{0}+1}^{2}({1\over k_{0}r_{0}})\over P_{1}^{2}({1\over k_{0}r_{0}})}={P_{l_{0}+1,1}^{2}\over P_{1,1}^{2}}+{2P_{l_{0}+1,1}P_{l_{0}+1,3}\over P_{1,1}^{2}k_{0}^{2}r_{0}^{2}}+O({1\over k_{0}^{4}r_{0}^{4}})$\\

It follows that, for $l_{0}$ even, and large $\{k_{0},k_{1}\}$;\\

$<U_{em,l_{0},k_{0}}^{Q}>-<U_{em,l_{0},k_{1}}^{Q}>\simeq {2Q^{2}Q_{l_{0},0}Q_{l_{0},2}\beta_{l_{0}}\over 1024\pi^{8}\epsilon_{0}Q_{0}^{2}r_{0}^{3}(k_{0}^{2}-k_{1}^{2})}$\\

$\simeq {Q^{2}Q_{l_{0},0}Q_{l_{0},2}\beta_{l_{0}}\over 128\pi^{10}\epsilon_{0}Q_{0}^{2}r_{0}(m_{0}^{2}-m_{1}^{2})}$\\

and for $l_{0}$ odd, and large $\{k_{0},k_{1}\}$;\\

$<U_{em,l_{0},k_{0}}^{Q}>-<U_{em,l_{0},k_{1}}^{Q}>\simeq {2Q^{2}P_{l_{0}+1,1}P_{l_{0}+1,3}\beta_{l_{0}}\over 1024\pi^{8}\epsilon_{0}P_{1,1}^{2}r_{0}^{3}(k_{0}^{2}-k_{1}^{2})}$\\

$\simeq {Q^{2}P_{l_{0}+1,1}P_{l_{0}+1,3}\beta_{l_{0}}\over 128\pi^{10}\epsilon_{0}P_{1,1}^{2}r_{0}(m_{0}^{2}-m_{1}^{2})}$\\

where $k_{0}\simeq {\pi\over r_{0}}(n_{0}+{l_{0}\over 2})$ and $m_{0}=2n_{0}+l_{0}$, $k_{1}\simeq {\pi\over r_{0}}(n_{1}+{l_{0}\over 2})$ and $m_{1}=2n_{1}+l_{0}$, with $\{m_{1},m_{2}\}\subset{\mathcal{Z}}_{\geq 1}$\\

which agrees closely with the Balmer series as claimed. Observe that for distinct $(l_{0},k_{0})$ and $(l_{1},k_{1})$, using the representation $(\sharp\sharp)$ and the orthogonality of the series $\delta_{l,m,k}$, that for $\{\alpha_{0},\beta_{0}\}$ and $\{\alpha_{1},\beta_{1}\}$ supported on $k_{0}\in S_{l_{0}}$ and $k_{1}\in S_{l_{1}}$ respectively, that;\\

$\int_{B(r_{0})}\overline{E}_{l_{0},k_{0}}^{\alpha_{0},\beta_{0}}\overline{\overline{E}_{l_{1},k_{1}}^{\alpha_{1},\beta_{1}}}d\overline{x}=0$ $(\sharp\sharp\sharp\sharp)$\\

For any $\overline{E}$ represented as in $(\dag\dag\dag\dag\dag\dag)$ we have that;\\

$\overline{E}=\sum_{l\geq 0}\sum_{k\in S_{l}}\overline{E}_{k,l}^{\alpha_{k},\beta_{k}}$\\

where $\alpha_{k}$ and $\beta_{k}$ are the restrictions of $\alpha$ and $\beta$ to $k\in S_{l}$. It follows from $(\sharp\sharp\sharp),(\sharp\sharp\sharp\sharp)$, that;\\

$U_{em}=\int_{B(r_{0})}|\overline{E}|^{2}d\overline{x}$\\

$=\sum_{l\geq 0}\sum_{k\in S_{l}}\int_{B(r_{0})}|\overline{E}_{k,l}^{\alpha_{k},\beta_{k}}|^{2}d\overline{x}$\\

$=\sum_{l\geq 0}\sum_{k\in S_{l}}U_{em,k,l}^{Q_{k,l}^{\alpha_{k},\beta_{k}}}$\\

where $Q_{k,l}^{\alpha_{k},\beta_{k}}=\int_{B(r_{0})}\rho_{k,l}^{\alpha_{k},\beta_{k}}d\overline{x}$\\

\end{proof}

\begin{rmk}
\label{ionisation}
Note that the condition $Q=0$ places no restriction on the values of $\alpha$ and $\beta$, when $l_{0}=1$. As the values of $\alpha$ and $\beta$ can vary continuously, this suggests that the quantisation phenomenon, observed in the previous lemma, occurs only when the atom is ionised, in which case $Q\neq 0$ and we observe the behaviour of the Balmer series. This point of view is supported by the results of the Franck-Hertz experiment.

\end{rmk}
\end{section}


\begin{thebibliography}{99}
\bibitem{BK} Vector Analysis, Bourne and Kendall, Oldbourne Mathematical Series, (1967).\\
\bibitem{BP} Elementary Differential Equations and Boundary Value Problems, Seventh Edition William Boyce and Richard DiPrima, John Wiley and Sons, (2001)\\
\bibitem{lcl} Electromagnetic Fields and Waves, Third Edition, Dale Corson, Francois Lorrain and Paul Lorrain, W.H Freeman and Company, (1988).\\
\bibitem{JD} Basics of Bessel Functions, Joella Deal, University Honor Theses, Portland State University, (2018).\\
\bibitem{E} Partial Differential Equations, Lawrence Evans, AMS Graduate Studies in Mathematics, (1998).\\
\bibitem{DV} Introductory Quantum Physics and Relativity, Jacob Dunningham and Vlatko Vedral, Imperial College Press, (2011).\\
\bibitem{G} Introduction to Electrodynamics, David Griffiths, Pearson International, (2008).\\
\bibitem{Gr} Instructor's Solutions Manual, Introduction to Electrodynamics, Third Edition, David Griffiths, Pearson International, (2004).\\
\bibitem{m} Math Stackexchange, https://math.stackexchange.com/questions/3310890/integral-of-exp-over-the-unit-ball\\
\bibitem{L} Revisiting the orthogonality of Bessel functions of the first kind on an infinite interval, J. Ponce de Leon, European Journal of Physics, (2015).\\
\bibitem{M} Quantum Mechanics, Volume 1, Albert Messiah, North Holland Publishing, (1965)\\
\bibitem{NR} Quantum Mechanics, Sixth Edition, Jim Napolitano and Alastair Rae, CRC Press, Taylor and Francis, (2016).\\
\bibitem{wave} A Nonstandard Solution to the Wave Equation, Tristram de Piro, available at http://www.curvalinea.net, submitted to Annals of Pure and Applied Logic, (2020).\\
\bibitem{heat2} Computing the Distribution of Velocities of some Solutions to the Nonstandard Diffusion Equation, Tristram de Piro, available at http://www.curvalinea.net, (2019).\\
\bibitem{dep1} Nonstandard Analysis and Physics, Tristram de Piro, available at http://www.curvalinea.net, (2014).\\
\bibitem{heat} Nonstandard Methods for Solving the Heat Equation, Tristram de Piro, available at http://www.curvalinea.net, (2016).\\
\bibitem{SS} Fourier Analysis, An Introduction, Elias Stein and Rami Shakarchi, Princeton Lectures in Analysis 1, (2002).\\
\bibitem{W} A Treatise on the Theory of Bessel Functions, G.N.Watson, Cambridge University Press, (1922).\\

\end{thebibliography}
\end{document}